\documentclass[11pt]{article}
\usepackage[latin1]{inputenc}
\usepackage[T1]{fontenc}
\usepackage{amsmath,amssymb,amstext}
\usepackage{theorem}
\usepackage{a4wide}
\usepackage{multicol}
\usepackage[english]{babel}
\usepackage{enumerate}
\usepackage{eufrak}
\usepackage{color}
\usepackage{mathrsfs}

\newtheorem{theorem}{Theorem}

\newtheorem{lemma}[theorem]{Lemma}

\newenvironment{proof}[1][Proof]{\textbf{#1.} }{\ \rule{0.5em}{0.5em}}

\renewcommand{\geq}{\geqslant}

\def\0.5{{\frac{1}{2}}}

\newcommand\1{\leavevmode\hbox{\rm \small1\kern-0.35em\normalsize1}}

\renewcommand{\thefootnote}{\fnsymbol{footnote}}

\begin{document}

\renewcommand{\thefootnote}{\arabic{footnote}}

\begin{center}
{\large \textbf{Parameter Estimation for a partially observed
Ornstein-Uhlenbeck process with long-memory noise}}~\\[0pt]
Brahim El Onsy\footnote{%
Faculty of Sciences and Techniques - Marrakech, Cadi Ayyad University,
Morocco. E-mail: \texttt{brahim.elonsy@gmail.com}} Khalifa Es-Sebaiy%
\footnote{%
National School of Applied Sciences - Marrakesh, Cadi Ayyad University,
Morocco. E-mail: \texttt{k.essebaiy@uca.ma}} Frederi G. Viens \footnote{%
Dept. Statistics and Dept. Mathematics, Purdue University, 150 N. University
St., West Lafayette, IN 47907-2067, USA. E-mail: \texttt{viens@purdue.edu}}
\\[0pt]
\textit{Cadi Ayyad University and Purdue University }\\[0pt]
~\\[0pt]
\end{center}

\noindent \textbf{Abstract}: We consider the parameter estimation problem
for the Ornstein-Uhlenbeck process $X$ driven by a fractional
Ornstein-Uhlenbeck process $V$, i.e. the pair of processes defined by the
non-Markovian continuous-time long-memory dynamics $dX_{t}=-\theta
X_{t}dt+dV_{t};\ t\geq 0$, with $dV_{t}=-\rho V_{t}dt+dB_{t}^{H};\ t\geq 0$,
where $\theta >0$ and $\rho >0$ are unknown parameters, and $B^{H}$ is a
fractional Brownian motion of Hurst index $H\in (\frac{1}{2},1)$. We study
the strong consistency as well as the asymptotic normality of the joint
least squares estimator $\left( \widehat{\theta }_{T},\widehat{\rho }%
_{T}\right) $ of the pair $\left( \theta ,\rho \right) $, based either on
continuous or discrete observations of $\{X_{s};\ s\in \lbrack 0,T]\}$ as
the horizon $T$ increases to +$\infty $. Both cases qualify formally as
partial-observation questions since $V$ is unobserved. In the latter case,
several discretization options are considered. Our proofs of asymptotic
normality based on discrete data, rely on increasingly strict restrictions
on the sampling frequency as one reduces the extent of sources of
observation. The strategy for proving the asymptotic properties is to study
the case of continuous-time observations using the Malliavin calculus, and
then to exploit the fact that each discrete-data estimator can be considered
as a perturbation of the continuous one in a mathematically precise way,
despite the fact that the implementation of the discrete-time estimators is
distant from the continuous estimator. In this sense, we contend that the
continuous-time estimator cannot be implemented in practice in any naïve
way, and serves only as a mathematical tool in the study of the
discrete-time estimators' asymptotics.\newline



\section{Introduction}

\subsection{Context and background}

The question of drift parameter estimation for solutions of stochastic
differential equations driven by fractional Brownian noise goes back at
least as far as the seminar work of Kleptsyna and Le Breton in \cite{KL}
(also see prior references therein), where a maximum likelihood estimator
(MLE) was proposed. This work was a first genuine attempt to show how to
compute the MLE in practice in the regular case (self-similarity `Hurst'
parameter $H>1/2$), by relying on continuous-time observation of a single
path over a finite time interval, with strong consistency and asymptotic
normality as the horizon increases to infinity. This work was itself
motivated by -- and is an extension of -- now classical ideas of how to use
the Girsanov theorem to compute the MLE in the case of white noise, as
presented in the 1978 treatise \cite{LS}. The continuous-time data was also
invoked in \cite{KL} to justify that any diffusion-type parameter (any
constant multiplicative term in front of the equation's driving noise) would
then be directly observable. That observation remains generally true in many
contexts for continuous-time data, including when the noise is fractional
Brownian. It explains why so many authors since \cite{KL} have continued to
study the estimation of drift parameters for fractional-noise-driven
problems. We have listed some of these works below. Our paper inscribes
itself in this line of work, and can be viewed as a study of partial
observation questions, as we will explain shortly.

Generally speaking, this effort to understand drift estimation for ergodic
diffusions, even Gaussian ones, is of fundamental importance for any
quantitative study where mean-reverting quantities are believed to be
asymptotically stationary, and are either observed with noise or are
intrinsically stochastic. The best known class of examples, which also
encapsulates the question of whether or not the stochastic process of
interest is observed directly or indirectly, is that of stochastic
volatility in quantitative finance. Some of the original ideas on how to
estimate this volatility's drift parameters is given in the 2000 research
monograph \cite{FPS}. Similar models with fractional noise were introduced
in the late 1990's, as in \cite{CR} for continuous time, but to our
knowledge, their statistical estimation was left unexplored for more than
ten years. The broadest question, applicable in finance and other fields, is
to estimate jointly the drift, diffusion, and memory parameters for
fractional-noise-driven equations; it is typically non-trivial, and with the
exception of one study in \cite{B}, largely unresolved by bona fide
statistical means for non-self-similar discretely observed continuous-time
processes. We will not address this issue here. In the case of partially
observed data, we refer the interested reader to a study in the context of
high-frequency financial data, where a sequential Bayesian methodology is
combined with classical estimation techniques and a calibration method to
find $H$: see \cite{CV}. We also mention a method described in \cite{R} for
a similar model with partial observation, where it is shown that the
minimax-sense optimal estimator has a very slow convergence rate, and relies
on high-frequency data; it can be argued that when applied to bona fide
financial data, this estimator cannot be implemented without leaving the
realm where continuous-time semi-martingale models with long-memory
volatility are appropriate.

The present paper looks specifically at a difficulty which arises when one
single path is used to estimate more than one real drift-type parameter.
Questions of identifiability can arise in this context (see for example the
treatment of the Generalized Method of Moments (GMM) as described in \cite%
{NM}). Once such a question has been resolved, the main practical obstacle
to implementation is typically that of discretizing the data given by the
path, i.e. using only observations in discrete time. For the sake of
conciseness, we look at a specific situation where identifiability is
resolved explicitly and in a natural way, within the task of deriving the
continuous-time and discrete-time estimators, without having to rely on
abstract conditions which would provide a priori identifiability of the
vector of parameters. Our objective is to describe conditions under which
strongly consistent and asymptotically normal estimation can be established
quantitatively, based on discretely-observed data alone. Since access to and
analysis of continuous-time data is not typically a realistic assumption, we
will view any estimators based on continuous-time data as tools in the task
of deriving strong consistency and asymptotic normality for the
discrete-time estimators. Our strategy is to attempt to discretize the
former, which includes a need to approximate Riemann and stochastic
integrals. In the process, we discover some fundamental differences between
the discrete and continuous time estimators. In particular, we find that
while least squares (LS) estimation appears to be the best tool based on
continuous-time observation of a path, when converted to discrete time data,
the estimator's interpretation as a LS optimizer is lost, and a GMM
interpretation seems more appropriate. This introduction, including our
summary of results in Section \ref{Summary}, contains specific details
supporting these ideas.

We are largely motivated by the paper \cite{BPS}, which studies the
drift-estimation problem for the Ornstein-Uhlenbeck process driven itself by
another, unobserved Ornstein-Uhlenbeck process (OU-OU). Their work only
deals with an underlying white noise. Specifically, let $W$ be a standard
Brownian motion and let $\theta $ and $\rho $ be positive real parameters.
The OU-OU process is the solution of the following system%
\begin{equation}
\left\{
\begin{array}{l}
X_{0}=0;\quad dX_{t}=-\theta X_{t}dt+dV_{t},\quad t\geq 0; \\
V_{0}=0;\quad dV_{t}=-\rho V_{t}dt+dW_{t},\quad t\geq 0.%
\end{array}%
\right.   \label{OUOU}
\end{equation}%
Here one may consider that $X$ is observed, while $V$ is not, or conversely,
or that both processes are observed; which interpretation is used makes a
crucial difference, as we will see.

Since the quadratic variation of $V$ is $t$, the classical Girsanov theorem
implies that a natural candidate to estimate $\theta $ is the MLE (recall
that this idea was already contained in \cite{LS}), which can be easily
computed for this model: one gets
\begin{eqnarray}
\widehat{\theta }_{T} &=&\frac{-\int_{0}^{T}X_{t}\delta X_{t}}{%
\int_{0}^{T}X_{t}^{2}dt}  \label{theta^Brownian case} \\
&=&\frac{-X_{T}^{2}+T}{2\int_{0}^{T}X_{t}^{2}dt},
\label{thetaWithSquareDecomposition}
\end{eqnarray}%
In (\ref{theta^Brownian case}), the integral with respect to $X$ must be
understood in the Itô sense. Consequently, line (\ref%
{thetaWithSquareDecomposition}) follows from Itô's formula, and the fact
that $X$ too has quadratic variation equal to $t$. One also notes that this
estimator's construction is in fact non-dependent on the form of the
bounded-variation part in the definition of $V$, which can be interpreted as
a form of robustness of $\widehat{\theta }_{T}$ with respect to model
misspecification, although we are about to see that the behavior of $%
\widehat{\theta }_{T}$ depends heavily on $V$'s drift specification.

On the other hand, it is worth noticing that $\widehat{\theta }_{T}$
coincides formally with a least squares estimator (LSE). Indeed, by
interpreting $\int_{0}^{T}X_{t}\dot{X}_{t}dt$ as the Itô integral in (\ref%
{theta^Brownian case}), $\widehat{\theta }_{T}$ formally minimizes
\begin{equation*}
\theta \longrightarrow \int_{0}^{T}\left\vert \dot{X}_{t}+\theta
X_{t}\right\vert ^{2}dt.
\end{equation*}%
By reversing the roles of $V$ and $X$, one can obtain an estimator for $\rho
$ similar to (\ref{thetaWithSquareDecomposition}). However, if $V$ is
unobserved, one may recast that estimator by using the estimated value of $%
V_{t}$ based on $\widehat{\theta }_{T}$ and the observed path of $X$, based
on (\ref{OUOU}). In other words, one defines
\begin{equation*}
\widehat{\rho }_{T}=\frac{-\widehat{V}_{T}^{2}+T}{2\int_{0}^{T}\widehat{V}%
_{t}^{2}dt}
\end{equation*}%
where $\widehat{V}_{t}=X_{t}+\widehat{\theta }_{T}\int_{0}^{t}X_{t}dt$ for
every $t\leqslant T$. As it turns out, this joint estimator $\left( \widehat{%
\theta }_{T},\widehat{\rho }_{T}\right) $, based on continuous observations
of $X$ alone, \emph{does not} converge to $\left( \theta ,\rho \right) $,
but rather to an explicit rational function of the pair of unknown
parameters $\left( \theta ,\rho \right) $. This was proposed and proved in
\cite{BPS}, wherein a semimartingale approach was used to study the
asymptotic behavior. Specifically they showed

\begin{itemize}
\item \emph{strong consistency}: as $T\longrightarrow +\infty $, almost
surely,%
\begin{eqnarray}
&&\widehat{\theta }_{T}\longrightarrow \theta +\rho   \label{oldthetahat} \\
&&\widehat{\rho }_{T}\longrightarrow \frac{\theta \rho (\theta +\rho )}{%
(\theta +\rho )^{2}+\theta \rho }  \label{oldrhohat}
\end{eqnarray}

\item \emph{asymtpotic normality}: as $T\longrightarrow +\infty $,
\begin{equation*}
\sqrt{T}\left( \widehat{\theta }_{T}-(\theta +\rho ),\widehat{\rho }_{T}-%
\frac{\theta \rho (\theta +\rho )}{(\theta +\rho )^{2}+\theta \rho }\right)
\overset{law}{\longrightarrow }N(0,\Gamma )
\end{equation*}
\end{itemize}

\noindent where $\Gamma $ is a covariance matrix which has a explicit form
as a function of $\theta $ and $\rho $. While intuition gathered from the
full-observation case is in fact erroneous when $V$ is unobserved (the naïve
candidates for $\widehat{\theta }_{T}$ and $\widehat{\rho }_{T}$ lead to
modified limits rather than $\left( \widehat{\theta }_{T},\widehat{\rho }%
_{T}\right) \longrightarrow \left( \theta ,\rho \right) $~), nevertheless we
have the full picture for the asymptotic behavior of the MLEs/LSEs
associated with (\ref{OUOU}), at the minor cost of having to solve a
non-linear system of two equations to obtain consistent estimates of $\left(
\theta ,\rho \right) $.

In the present paper, our goal is to investigate what happens when, in (\ref%
{OUOU}), the standard Brownian motion $W$ is replaced by a fractional
Brownian motion $B^{H}$. Thus we assume from now on that $X$ is an
Ornstein-Uhlenbeck process driven by a fractional Ornstein-Uhlenbeck process
$V$ : this means the pair $\left( X,V\right) $ is the unique solution of the
system of linear stochastic differential equations
\begin{equation}
\left\{
\begin{array}{l}
X_{0}=0;\quad dX_{t}=-\theta X_{t}dt+dV_{t},\quad t\geq 0 \\
V_{0}=0;\quad dV_{t}=-\rho V_{t}dt+dB_{t}^{H},\quad t\geq 0,%
\end{array}%
\right.   \label{OUFOU}
\end{equation}%
where $B^{H}=\left\{ B_{t}^{H},t\geq 0\right\} $ is a fractional Brownian
motion (fBm) with Hurst index $H\in (\frac{1}{2},1)$, and where $\theta >0$
and $\rho >0$ are unknown parameters. Though $X$ is defined for all $H$ in $%
(0,1)$, to keep technical difficulties to a reasonable level, we restrict
ourselves to the case $H\in (1/2,1)$. It turns out that we need the
condition $\theta \neq \rho $ for identifiability; remarkably, this is not
needed when $H=1/2$, as we saw in the system (\ref{oldthetahat}), (\ref%
{oldrhohat}). Details of our results are summarized in Section \ref{Summary}.%
\vspace*{0.12in}

We now provide some references to estimation with fBm noise, which are
further motivations for our work. We mentioned that the single-drift
parameter estimation problem for fractional diffusion processes based on
continuous-time observations was originally studied in \cite{KL} via maximum
likelihood; more recent work on this question includes, e.g., \cite{TV,
Rao2010}. Recently, the LSE for the fractional Ornstein-Uhlenbeck (fOU)
process, i.e. the process $V$ in (\ref{OUFOU}) was proposed in \cite{HN}:
assuming $V$ is fully observed in continuous time, the LSE for $\rho $ is
defined by
\begin{equation*}
\overline{\rho }_{T}:=-\frac{\int_{0}^{T}{V}_{t}\delta {V}_{t}}{\int_{0}^{T}{%
V}_{t}^{2}dt},
\end{equation*}%
where the integral $\int_{0}^{T}{V}_{t}\delta {V}_{t}$ is interpreted in the
sense of Skorokhod. This integral is the extension to fBm of Itô's integral
for Brownian motion. In the case $\rho >0$, \cite{HN} proved that $\overline{%
\rho }_{T}$ is strongly consistent and asymptotically normal as $%
T\rightarrow \infty $. Unlike in the case of Brownian motion ($H=1/2$)
discussed above, $\overline{\rho }_{T}$ does not coincide with the MLE given
in \cite{KL}, because the Girsanov theorem for fBm takes a different form
than in the case $H=1/2$. Given the notorious fact that Skorohod integrals
are difficult to interpret in practical terms for fBm, the authors of \cite%
{HN} proposed in addition the following alternate estimator, which is
arguably a method of moments :%
\begin{equation}
\left( \frac{1}{H\Gamma (2H)T}{\int_{0}^{T}V_{t}^{2}dt}\right) ^{-\frac{1}{2H%
}};  \label{alternative estimator}
\end{equation}%
they proved it is strongly consistent and asymptotically normal. In the case
$\rho <0$, \cite{BEO} established that $\overline{\rho }_{T}$ of $\rho $ is
strongly consistent and asymptotically Cauchy-distributed.

The alternate choice of estimator (\ref{alternative estimator}) does not,
however, avoid the use of continuous observations over discrete ones; this
is a problem with many works on fBm-driven models, and is an additional
motivation for us to study the asymptotics of estimation for fBm-driven
processes based on discrete observations. There exists a rich literature on
this practical problem for ordinary diffusions driven by Brownian motions;
we refer for instance to \cite{Rao2010}. A handful of authors are beginning
to look at these questions with various fBm-driven models, starting with
\cite{TV} in 2007, and more recently \cite{AV, BI, khalifa, NT}. In
particular, for the fOU process $V$, motivated by the estimator given in (%
\ref{alternative estimator}), \cite{HS} studied its natural Riemann-sum-type
time discretization
\begin{equation}
\left( \frac{1}{nH\Gamma (2H)}\sum_{i=1}^{n}X_{{i}}^{2}\right) ^{-\frac{1}{2H%
}}  \label{alternative estimator discrete}
\end{equation}%
providing strongly consistency and Berry-Esséen-type theorems for it. While
we have no doubt that this estimator is indeed strongly consistent and
asymptotically normal, the proofs in \cite{HS} rely on a possibly flawed
technique, since the passage from line -7 to -6 on page 434 therein requires
the condition $H>3/4$, while one expects normal asymptotics only for the
case $H\leqslant 3/4$.\vspace*{0.12in}

In our paper, we focus our discussion on estimators which are derived from a
basic LSE, since that technique is known, at least in the Brownian case
described in \cite{BPS}, to allow for a straightforward bivariate extension,
as mentioned previously. The LSE has also given rise to a number of
successful studies in the univariate case with fractional processes: we have
already cited \cite{AV, BI, khalifa, NT}, while \cite{AM} is the
continuous-time version of \cite{AV}.

Herein, specifically, we begin our study of LSE for $\left( \theta ,\rho
\right) $ in (\ref{OUFOU}) by using the formal least-squares interpretation
mentioned above, i.e. looking for the minimizer of $\theta \longrightarrow
\int_{0}^{T}\left\vert \dot{X}_{t}+\theta X_{t}\right\vert ^{2}dt$; this
leads formally to the following estimator for $\theta $
\begin{equation}
\widehat{\theta }_{T}=-\frac{\int_{0}^{T}X_{t}\delta X_{t}}{%
\int_{0}^{T}X_{t}^{2}dt}  \label{theta^fractional case}
\end{equation}%
and to the similar estimator
\begin{equation}
\widehat{\rho }_{T}=-\frac{\int_{0}^{T}\widehat{V}_{t}\delta \widehat{V}_{t}%
}{\int_{0}^{T}\widehat{V}_{t}^{2}dt}  \label{rho^fractional case}
\end{equation}%
for $\rho $, where, because of the fact that $V$ is unobserved, one uses $%
\widehat{V}_{t}=X_{t}+\widehat{\theta }_{T}\int_{0}^{t}X_{t}dt$ for every $%
t\leqslant T$, instead of relying on the unobserved $V_{t}$ in the
construction of $\widehat{\rho }_{T}$. These estimators $\widehat{\theta }%
_{T},\widehat{\rho }_{T}$ are no longer the MLEs, since, as we mentioned,
the Girsanov theorem for fBm does not have the same form as for Brownian
motion, but they are still formally LSEs. Nevertheless, there is a major
difference with respect to the Brownian motion case. Indeed, since the
process $X$ is no longer a semimartingale, in (\ref{theta^fractional case})
and (\ref{rho^fractional case}) one cannot interpret the numerators using
the Itô integral; the Skorohod integral turns out to be the correct notion
to use here. We mentioned above that Skorohod integrals are difficult to use
in practice, but since our Hurst parameter $H$ exceeds $1/2$, it is possible
to reinterpret the Skorohod integrals as so-called Young integrals, a
pathwise notion, modulo a correction term which we will be able to compute
explicitly thanks to the Malliavin calculus.

Having succeeded in correctly interpreting the stochastic integrals in (\ref%
{theta^fractional case}) and (\ref{rho^fractional case}), the issue of how
to discretize them becomes paramount to practical implementation, and herein
we will propose several different options, some of which allow for strong
consistency and asymptotic normality under broader conditions than others.

Our discrete-observation study also applies to the case $H=1/2$ as a
limiting case. The article \cite{BPS} treats this case solely with
continuous observations; our work thus covers an extension of their work to
discrete observations. Checking the validity of this statement rigorously is
straightforward; for the sake of conciseness, we leave it to the interested
reader.

\subsection{Summary of results and heuristics\label{Summary}}

We now summarize our results, the structure of our article, and our main
proof elements, including useful heuristics when available.

\begin{itemize}
\item In Section 2 and in the Appendix we introduce the needed mathematical
background material for our study, including elements of the Malliavin
calculus, a convenient criterion for establishing normal convergence on
Wiener chaos, and the relation between Skorohod integrals and Young
integrals with respect to fBm when $H>1/2$.

\item In Section 3, we concentrate on proving strong consistency and
asymptotic normality for the estimators $\widehat{\theta }_{T}$ and $%
\widehat{\rho }_{T}$ with continuous observations.

\begin{itemize}
\item We first prove the following almost surely convergences:%
\begin{eqnarray}
\widehat{\theta }_{T} &\rightarrow &\theta ^{\ast }:=\rho +\theta ,
\label{newthetahat} \\
\widehat{\rho }_{T} &\rightarrow &\rho ^{\ast }:=\frac{\rho \theta \left(
\theta +\rho \right) }{\frac{\rho ^{2-2H}-\theta ^{2-2H}}{\theta ^{-2H}-\rho
^{-2H}}+(\theta +\rho )^{2}}.  \label{newrhohat}
\end{eqnarray}%
The proof relies on studying the numerator and the denominator of the
expressions for $\widehat{\theta }_{T}$ and $\widehat{\rho }_{T}$
separately. For the denominators, we rely on Birkhoff's ergodic theorem, and
elementary covariance estimations for exponential convolutions with fBm. For
the numerators, we express the Skorohod integrals as Young integrals plus
their correction terms involving Malliavin derivatives which are explicit
deterministic functions since our processes are Gaussian.

\item The expression $\theta ^{\ast }=\rho +\theta $ in (\ref{newthetahat})
is easy to explain: as noted below in line (\ref{SIDE}), $X$ satisfies a
stochastic integro-differential equation in which the zero-mean-reversion
term $-\left( \rho +\theta \right) X\left( t\right) dt$ appears, and thus a
natural candidate for a consistent estimator of $\rho +\theta $ is the LSE $%
\widehat{\theta }$, whether one adds merely one mean-zero noise term $dB^{H}$
or another term $\left( \int_{0}^{t}X_{s}ds\right) $ which is asymptotically
small. This is why the limiting behavoir of $\widehat{\theta }$ remains the
same for us as in (\ref{oldthetahat}), which is the Brownian case $(H=1/2$)
studied in \cite{BPS}. The details of this heuristic are omitted, since the
full proof we present herein is needed to be convincing. On the other hand,
the expression in (\ref{newrhohat}) for our $\rho ^{\ast }$ is more opaque;
there does not seem to be a direct heuristic to explain it, beyond our
computations. When one compares our $\rho ^{\ast }$ in (\ref{newrhohat})
with the $\rho ^{\ast }$ in (\ref{oldrhohat}) found in \cite{BPS}, one sees
that the term $\rho \theta $ in (\ref{oldrhohat}) is replaced by the
expression $\left( \rho ^{2-2H}-\theta ^{2-2H}\right) /\left( \theta
^{-2H}-\rho ^{-2H}\right) $, which can help identify how the case of fBm
deviates from the case $H=1/2$.\footnote{%
As a way to compare these terms, which do coincide when $H=1/2$, we can see
that if $\rho $ tends to $\theta $, the aforementioned expression in (\ref%
{oldrhohat}) tends to $\theta ^{2}\left( 1-H\right) /H$, which thus deviates
significantly from the case $H=1/2$ quantitatively, particularly for $H$
close to $1$.} The expression for $\rho ^{\ast }$ is analyzed further in the
context of discretizing $\left( \widehat{\theta },\widehat{\rho }\right) $,
which helps explain to some extent why this complicated expression arises,
as the reader will find out in the first paragraph of Section \ref{XandSigma}%
.

\item We prove asymptotic normality of $\left( \widehat{\theta }_{T},%
\widehat{\rho }_{T}\right) $ by expressing the Skorohod integrals as
iterated Wiener integrals, identifying dominant portions of these integrals,
relying on a criterion for normal convergence in law in Wiener chaos,
combined with a number of almost sure convergences. Our main asymptotics
normality result is a central limit theorem that holds for $\sqrt{T}\left(
\widehat{\theta }_{T}-\theta ^{\ast },\widehat{\rho }_{T}-\rho ^{\ast
}\right) $ as $T\rightarrow \infty $, as soon as $H\in \lbrack 1/2,3/4)$.
The asymptotic covariance is given explicitly. See Theorem \ref{convergence
in law theorem}. The upper limit of validity of this theorem is a typical
threshold in normal convergence theorems in the second Wiener chaos. See for
example a classical instance of this situation in the Breuer-Major central
limit theorem, as presented in \cite[Chapter 7]{NP-book}. For $H>3/4$, we
conjecture that the estimators are asymptotically Rosenblatt-distributed
(again see \cite[Chapter 7]{NP-book} for a classical example of such a
phenomenon), and that the convergence occurs almost surely; this point is
not discussed further, for the sake of conciseness.
\end{itemize}

\item The topic of Section 4 is to construct estimators based solely on
discrete observations. The asymptotic results we prove still require
increasing horizon. We also assume that $X$ is observed at evenly spaced
intervals, with a time step $\Delta _{n}$, and we set the time horizon to be
$T_{n}=n\Delta _{n}\rightarrow \infty $ as $n\rightarrow \infty $. Let $%
t_{k}=k\Delta _{n}$ be the $k$th observation time. For instance, the case $%
\Delta _{n}=1$ corresponds to a fixed observation frequency; other
conditions on $\Delta _{n}$ will include requiring $\Delta _{n}$ to tend to $%
0$ as fast as a certain negative power of $n$, i.e. the observation
frequency increases as the horizon increases. For some strong consistency
results, it will even be possible for us to relax conditions on $\Delta _{n}$
where it is allowed to tend to infinity like a power of $n$, i.e. with
decreasing frequency as the horizon increases.

Arguably, to be consistent with the assumption that only $X$ is observed,
the only estimators which are of practical use are those which rely solely
on the values $\left\{ X_{t_{k}}:k=1,\ldots ,n\right\} $. Designing such an
estimator by discretizing $\left( \widehat{\theta }_{T},\widehat{\rho }%
_{T}\right) $ turns out to be a difficult task, in which the final
expression solves a non-linear system in the spirit of that which would
follow from (\ref{newthetahat}) and (\ref{newrhohat}), but is rather
distinct from this system because of the difficulty in how to interpret the
discretizations of the Skorohod or Young integrals.

The method we have chosen moves through several intermediate steps, where we
gradually increase the number of terms in the estimators which are replaced
by discretized versions. This method has the advantage of clearly showing
where the restrictions on the observation frequency $\Delta _{n}$ come into
play. Each intermediate estimator can be considered as a perturbation of the
previous one, starting with the continuous-time estimator of Section 3. Thus
arguably all these estimators can be considered as tools used for the final
objective, attained in Section \ref{Xobserved}, of constructing a strongly
consistent and asymptotically normal estimator of $\left( \theta ,\rho
\right) $ based only on the data $\left\{ X_{t_{k}}:k=1,\ldots ,n\right\} $.
Nonetheless, some of the other estimators are relevant in their own right,
as they might correspond to realistic partial or full observation cases.

\begin{itemize}
\item The main technical estimates which allow our discretization are given
at the beginning of Section 4. These are Lemmas \ref{cv Q(X)} and \ref{cv
Q(Sigma)}, based on applications of the Borel-Cantelli lemma. For $Z$ any
stochastic process, we let
\begin{equation*}
Q_{n}\left( Z\right) :=n^{-1}\sum_{k=1}^{n}\left( Z_{t_{k}}\right) ^{2}.
\end{equation*}%
Let $S_{t}:=\int_{0}^{t}X_{s}^{2}ds$ and $\Sigma _{t}:=\int_{0}^{t}X_{s}ds$.
We show that the discrepancy between $S_{T_{n}}/T_{n}$ and its discrete
version $Q_{n}\left( X\right) $ is $=o\left( 1/\sqrt{T_{n}}\right) $ almost
surely. We then compute three different discrepancies related to $\Sigma $:
first we show that the difference between $T_{n}^{-1}\int_{0}^{T_{n}}\Sigma
_{t}^{2}dt$ and its discrete version $Q_{n}\left( \Sigma \right) $ is also $%
=o\left( 1/\sqrt{T_{n}}\right) $ almost surely. Then we show that with $\hat{%
\Sigma}$ the version of $\Sigma $ which depends only on $X$ observations,
i.e. $\widehat{\Sigma }_{t_{k}}:=\Delta_n\sum_{i=1}^{k}X_{t_{i-1}}^{2}$, we
get that $Q_{n}\left( \hat{\Sigma}\right) -Q_{n}\left( \Sigma \right) $ tend
to $0$ almost surely. This is helpful to prove strong consistency of
discrete estimators. To prove asymptotic normality, we need that $%
Q_{n}\left( \hat{\Sigma}\right) -Q_{n}\left( \Sigma \right) =o\left( 1/\sqrt{%
T_{n}}\right) $, which we prove holds almost surely. Increasingly
restrictive conditions on $\Delta _{n}$ are needed for these successive
results.

\item We first concentrate on discretizing the denominators of $\left(
\widehat{\theta }_{T_{n}},\widehat{\rho }_{T_{n}}\right) $.

\begin{itemize}
\item We replace the denominator of $\widehat{\theta }_{T_{n}}$ by $%
Q_{n}\left( X\right) $, yielding an estimator $\tilde{\theta}_{n}$, and we
then replace the denominator of $\widehat{\rho }_{T}$ by $Q_{n}\left(
X\right) +(\tilde{\theta}_{n})^{2}Q_{n}\left( \Sigma \right) $, yielding an
estimator $\tilde{\rho}_{n}$, because, as it turns out, $\int_{0}^{T}%
\widehat{V}_{t}^{2}dt$ is asymptotically equivalent, almost surely, to $%
Q_{n}\left( X\right) +(\tilde{\theta}_{n})^{2}Q_{n}\left( \Sigma \right) $.
Thanks to this, to the almost sure equivalence of $Q_{n}\left( X\right) $
with $S_{T_{n}}/T_n$, and similarly for $Q_{n}\left( \Sigma \right) $,
coming from Lemmas \ref{cv Q(X)} and \ref{cv Q(Sigma)}, the strong
consistency and asymptotic normality of $\left( \tilde{\theta}_{n},\tilde{%
\rho}_{n}\right) $ follows from that of $\left( \widehat{\theta }_{T_{n}},%
\widehat{\rho }_{T_{n}}\right) $ proved in Section 3. Here it is sufficient
to assume $H\in \left( 1/2,1\right) $ and $\Delta _{n}\leqslant n^{\alpha }\
$for some $\alpha \in (-\infty ,1/H)$ for the strong consistency; note that $%
\Delta _{n} $ is allowed to remain constant or even increase like a moderate
power in this case. For the asymptotic normality, it is sufficient that $%
H\in \left( 1/2,3/4\right) $ and $n\Delta _{n}^{H+1}\rightarrow 0$.

\item A second result is obtained in which we forego having access to the
process $\Sigma $ itself, relying instead on its discrete version $\widehat{%
\Sigma }_{t_{k}}=\Delta _{n}\sum_{i=1}^{k}X_{t_{i-1}}^{2}$; in this case,
almost-sure converge of $\tilde{\rho}_{n}$ requires $n^{\alpha +1}\Delta
_{n}^{H+1}\rightarrow 0$ for some $\alpha >0$, and the central-limit result
for $\tilde{\rho}_{n}$ requires $n^{3}\Delta _{n}^{2H+3}\rightarrow 0$.
\end{itemize}

\item We are then able to define and study a bonafide estimator based on
discrete data alone.

\begin{itemize}
\item We begin with assuming that we have access to both $X_{t_{k}}$ and $%
\Sigma _{t_{k}}$ for all $k=1,\ldots ,n$. The stochastic integrals in $%
\tilde{\theta}_{n}$ and $\tilde{\rho}_{n}$ were analyzed in Section 3, and
were found, under scaling by $T_{n}^{-1}$, to be asymptotically constant,
where the explicit constants depend on the parameters. By using these limits
and a discretization of the Riemann integrals in the denominators of $\tilde{%
\theta}_{n}$ and $\tilde{\rho}_{n}$, this allows us, at the beginning of
Section \ref{XandSigma}, to motivate the definition of a pair of estimators $%
\left( \check{\theta}_{n},\check{\rho}_{n}\right) $ as solution of the
non-linear system
\begin{equation*}
F\left( \check{\theta}_{n},\check{\rho}_{n}\right) =\left(
Q_{n}(X),Q_{n}(\Sigma )\right)
\end{equation*}%
where $F$ is a positive function of the variables $\left( x,y\right) $ in $%
(0,+\infty )^{2}$ defined by: for every $(x,y)\in (0,+\infty )^{2}$
\begin{equation*}
F(x,y)=H\Gamma (2H)\times \left\{
\begin{array}{lcl}
\frac{1}{y^{2}-x^{2}}\left( y^{2-2H}-x^{2-2H},x^{-2H}-y^{-2H}\right) \quad %
\mbox{if }\ x\neq y &  &  \\
\left( (1-H)x^{-2H},Hx^{-2H-2}\right) \quad \mbox{if }\ y=x, &  &
\end{array}%
\right.
\end{equation*}%
and the data statistics used in the system are $Q_{n}\left( X\right) $ and $%
Q_{n}\left( {\Sigma }\right) $. Strong consistency and asymptotic normality
follow for the uniquely defined $\left( \check{\theta}_{n},\check{\rho}%
_{n}\right) $. The delicate computation of the asymptotic covariance is
given. The parameter restrictions remain the same as for $\left( \tilde{%
\theta}_{n},\tilde{\rho}_{n}\right) $, namely strong consistency if $H\in
\left( 1/2,1\right) $ and $\Delta _{n}\leqslant n^{\alpha }\ $for some $%
\alpha \in (-\infty ,1/H)$, and asymptotic normality if $H\in \left(
1/2,3/4\right) $ and $n\Delta _{n}^{H+1}\rightarrow 0$.

\item By redefining the estimators $\left( \check{\theta}_{n},\check{\rho}%
_{n}\right) $ using $\hat{\Sigma}$ instead of $\Sigma $, one ensures that
only the data $\left\{ X_{t_{k}}:k=1,\ldots ,n\right\} $ is used. Here, the
results from the previous case can be applied directly with the auxiliary
results on how $\hat{\Sigma}$ perturbs $\Sigma $ (Lemma \ref{cv Q(Sigma)}),
to obtain the same almost-sure convergence and central-limit result, but
these are now restricted respectively to the aforementioned observation
frequency parameter ranges $n^{\alpha +1}\Delta _{n}^{H+1}\rightarrow 0$ for
some $\alpha >0$, and $n^{3}\Delta _{n}^{2H+3}\rightarrow 0$.
\end{itemize}

\item Finally, to illustrate how the complexity of the nonlinearities in the
definition of $\left( \check{\theta}_{n},\check{\rho}_{n}\right) $ may be
attributable to the partial-observation problem, we define a pair of
estimators $\left( \underline{\theta }_{n},\underline{\rho }_{n}\right) $
under the assumption that both $\left\{ X_{t_{k}}:k=1,\ldots ,n\right\} $
and $\left\{ V_{t_{k}}:k=1,\ldots ,n\right\} $ are available. The $%
\underline{\rho }_{n}$ is explicit given $\left\{ V_{t_{k}}:k=1,\ldots
,n\right\} $, and is identical to the one given in \cite{HS}, i.e. (\ref%
{alternative estimator discrete}), as it should be. The $\underline{\theta }%
_{n}$ satisfies the following straightforward non-linear equation given $%
\underline{\rho }_{n}$and $\left\{ X_{t_{k}}:k=1,\ldots ,n\right\} $:%
\begin{equation*}
\left( \underline{\theta }_{n}\right) ^{2-2H}-\left( \frac{Q_{n}(X)}{H\Gamma
(2H)}\right) \left( \underline{\theta }_{n}\right) ^{2}=\left( \underline{%
\rho }_{n}\right) ^{2-2H}-\left( \frac{Q_{n}(X)}{H\Gamma (2H)}\right) \left(
\underline{\rho }_{n}\right) ^{2}.
\end{equation*}%
The parameter restrictions are the same as when $X$ and $\Sigma $ are
discretely observed: strong consistency holds for $\left( \underline{\theta }%
_{n},\underline{\rho }_{n}\right) $ if $H\in \left( 1/2,1\right) $ and $%
\Delta _{n}\leqslant n^{\alpha }\ $for some $\alpha \in (-\infty ,1/H)$, and
asymptotic normality holds if $H\in \left( 1/2,3/4\right) $ and $n\Delta
_{n}^{H+1}\rightarrow 0$.
\end{itemize}
\end{itemize}

Before we proceed with the details of our study, we provide needed
mathematical tools in Section 2.

\section{Preliminaries}

In this section we describe some basic facts on the stochastic calculus with
respect to a fractional Brownian motion. For a more complete presentation on
the subject, see \cite{nualart-book} and \cite{AN}.\newline
The fractional Brownian motion $(B_{t}^{H},t\geq 0)$ with Hurst parameter $%
H\in (0,1)$, is defined as a centered Gaussian process starting from zero
with covariance
\begin{equation*}
R_{H}(t,s)=E(B_{t}^{H}B_{s}^{H})=\frac{1}{2}\left(
t^{2H}+s^{2H}-|t-s|^{2H}\right) ;\ s,\ t\geq ~0,
\end{equation*}%
We assume that $B^{H}$ is defined on a complete probability space $(\Omega ,%
\mathcal{F},P)$ such that $\mathcal{F}$ is the sigma-field generated by $%
B^{H}$. By Kolmogorov's continuity criterion and the fact $E\left(
B_{t}^{H}-B_{s}^{H}\right) ^{2}=|s-t|^{2H}$, we deduce that $B^{H}$ {admits
a version which} has Hölder continuous paths of any order $\gamma <H$.

Fix a time interval $[0,T]$. We denote by $\mathcal{H}$ the canonical
Hilbert space associated to the fractional Brownian motion $B^{H}$; the book
\cite{nualart-book}, among many other references, can be consulted for the
construction and properties of $\mathcal{H}$. We use the following
convenient notation for Wiener integrals with respect to $B^{H}$:%
\begin{equation*}
B^{H}\left( \varphi \right) :=\int_{0}^{T}\varphi \left( s\right) dB^{H}.
\end{equation*}%
Of interest to us is the fact that, with $H>1/2$, for a pair of (non-random)
functional elements $\varphi ,\psi $ of $\mathcal{H}$, its inner product
satisfies%
\begin{equation*}
\left\langle \varphi ,\psi \right\rangle _{\mathcal{H}}=E\left( B^{H}\left(
\varphi \right) B^{H}\left( \psi \right) \right)
=H(2H-1)\int_{0}^{T}\int_{0}^{T}\varphi (u)\psi (v)|u-v|^{2H-2}dudv.
\end{equation*}%
It follows from \cite{PT} that the set $|\mathcal{H}|$ of functional
elements in $\mathcal{H}$ is Banach and actually contains $L^{\frac{1}{H}%
}([0,T])$.

The Malliavin derivative $D$ w.r.t. $B^{H}$, which is an $\mathcal{H}$%
-values operator, is defined first by setting that
\begin{equation*}
DB^{H}\left( \varphi \right) =\varphi
\end{equation*}%
for any $\varphi \in \mathcal{H}$, and then by requiring that it satisfy a
multi-parameter chain rule: for any $f\in \mathrm{C}_{b}^{\infty }({\
\mathbb{R}}^{n},{\ \mathbb{R}})$ (infinitely differentiable functions from ${%
\mathbb{R}}^{n}$ to ${\mathbb{R}}$ with bounded partial derivatives) and any
$\varphi _{1},...,\varphi _{n}\in \mathcal{H}$, $D$ operates on the cylinder
r.v. $F:=$ $f(B^{H}(\varphi _{1}),...,B^{H}(\varphi _{n}))$ as {%
\begin{equation*}
DF=\sum_{i=1}^{n}\frac{\partial f}{\partial x_{i}}(B^{H}(\varphi
_{1}),...,B^{H}(\varphi _{n}))\varphi _{i}.
\end{equation*}%
} The domain $D^{1,2}$ of $D$ is then the the closure of the set of cylinder
r.v.'s $F$ with respect to the norm{\
\begin{equation*}
\Vert F\Vert _{1,2}^{2}:=E(F^{2})+E(\Vert DF\Vert _{{\mathcal{H}}}^{2}).
\end{equation*}%
} The divergence operator $\delta $ is the adjoint of the derivative
operator $D$ : an ${\mathcal{H}}$-valued r.v. $u\in L^{2}(\Omega ;\mathcal{H}%
)$ belongs to its domain $Dom\delta $ if
\begin{equation*}
E\left\vert \langle DF,u\rangle _{\mathcal{H}}\right\vert \leqslant
c_{u}\Vert F\Vert _{L^{2}(\Omega )}
\end{equation*}%
for some constant $c_{u}$ and every cylinder r.v. $F$. In this case $\delta
(u)$ is uniquely defined by the duality
\begin{equation*}
E(F\delta (u))=E\left\langle DF,u\right\rangle _{\mathcal{H}}
\end{equation*}%
for any $F\in D^{1,2}$. We will make use of the notation
\begin{equation*}
\delta (u)=\int_{0}^{T}u_{s}\delta B_{s}^{H},\quad u\in Dom\delta .
\end{equation*}%
In particular, $\delta $ extends the Wiener integral: for $h\in \left\vert
\mathcal{H}\right\vert $, $B^{H}(h)=\delta (h)=\int_{0}^{T}h_{s}\delta
B_{s}^{H}.$\newline
For every $n\geq 1$, let ${\mathcal{H}}_{n}$ be the nth Wiener chaos of $%
B^{H}$, that is, the closed linear subspace of $L^{2}(\Omega )$ generated by
the random variables $\{H_{n}(B^{H}(h)),h\in {{\mathcal{H}}},\Vert h\Vert _{{%
\mathcal{H}}}=1\}$ where $H_{n}$ is the $n$th Hermite polynomial. The
mapping ${I_{n}(h^{\otimes n})}=n!H_{n}(B^{H}(h))$ provides a linear
isometry between the symmetric tensor product ${\mathcal{H}}^{\odot n}$
(equipped with the modified norm $\Vert .\Vert _{{\mathcal{H}}^{\odot n}}=%
\frac{1}{\sqrt{n!}}\Vert .\Vert _{{\mathcal{H}}^{\otimes n}}$) and ${%
\mathcal{H}}_{n}$. It also turns out that ${I_{n}(h^{\otimes n})}$ is the
multiple Wiener integral of ${h^{\otimes n}}$ w.r.t. $B^{H}$. For every $%
f,g\in {{\mathcal{H}}}^{\odot n}$ the following product formula holds
\begin{equation*}
E\left( I_{n}(f)I_{n}(g)\right) =n!\langle f,g\rangle _{{\mathcal{H}}%
^{\otimes n}}.
\end{equation*}%
For $h\in {\mathcal{H}}^{\otimes n}$, the multiple Wiener integrals $I_{q}(f)
$, which exhaust the set ${\mathcal{H}}_{q}$, satisfy a hypercontractivity
property (equivalence in ${\mathcal{H}}_{q}$ of all $L^{p}$ norms for all $%
p\geq 2$), which implies that for any $F\in \oplus _{l=1}^{q}{\mathcal{H}}%
_{l}$, we have
\begin{equation}
\left( E\big[|F|^{p}\big]\right) ^{1/p}\leqslant c_{p,q}\left( E\big[|F|^{2}%
\big]\right) ^{1/2}\ \mbox{ for any }p\geq 2.  \label{hypercontractivity}
\end{equation}%
It is well-known that $L^{2}(\Omega )$ can be decomposed into the infinite
orthogonal sum of the spaces ${\mathcal{H}}_{n}$. That is, any square
integrable random variable $F\in L^{2}(\Omega )$ admits the following
\textquotedblleft Wiener chaos\textquotedblright\ expansion
\begin{equation*}
F=E(F)+\sum_{n=1}^{\infty }I_{n}(f_{n}),
\end{equation*}%
where the $f_{n}\in {{\mathcal{H}}}^{\odot n}$ are uniquely determined by $F$%
. \newline
Finally, we will use the following central limit theorem for multiple
stochastic integrals (see \cite{NO}).

\begin{theorem}
\label{NO} Let $\{F_{n}\,,n\geq 1\}$ be a sequence of random variables in
the $q$-th Wiener chaos ${\mathcal{H}}_{q}$, $q\geq 2$, such that $%
\lim_{n\rightarrow \infty }E(F_{n}^{2})=\sigma ^{2}$. Then the following
conditions are equivalent:

\begin{itemize}
\item[(i)] $F_n$ converges in law to $\mathcal{N}(0,\sigma^2)$ as $n$ tends
to infinity.

\item[(ii)] $\|DF_n\|^2_{\mathcal{H}} $ converges in $L^2$ to a constant as $%
n$ tends to infinity.
\end{itemize}
\end{theorem}

\section{Asymptotic behavior of LSEs}

Throughout the paper we assume that $H\in (\frac{1}{2},1)$, $\theta >0$ and $%
\rho >0$ such that $\theta \neq \rho $.\newline
It is readily checked that we have the following explicit expression for $%
X_{t}$:
\begin{equation}
X_{t}=\frac{\rho }{\rho -\theta }X_{t}^{\rho }+\frac{\theta }{\theta -\rho }%
X_{t}^{\theta }  \label{representationX}
\end{equation}%
where for $m>0$
\begin{equation}
X_{t}^{m}=\int_{0}^{t}e^{-m(t-s)}dB_{s}^{H}.  \label{Xm}
\end{equation}%
On the other hand, we can also write that the system (\ref{OUFOU}) implies
that $X$ solves the following stochastic integro-differential equation%
\begin{equation}
dX_{t}=-\left( \theta +\rho \right) X_{t}dt-\rho \theta \left(
\int_{0}^{t}X_{s}ds\right) dt+dB_{t}^{H}.  \label{SIDE}
\end{equation}%
For convenience, and because it will play an important role in the
forthcoming computations, we introduce the following processes related to $%
X_{t}$:
\begin{equation*}
S_{T}=\int_{0}^{T}X_{t}^{2}dt;\quad \Sigma _{T}=\int_{0}^{T}X_{t}dt;\quad
L_{T}=\int_{0}^{T}V_{t}^{2}dt;\quad P_{T}=\int_{0}^{T}X_{t}V_{t}dt;
\end{equation*}%
and
\begin{equation*}
\widehat{L}_{T}=\int_{0}^{T}\widehat{V}_{t}^{2}dt
\end{equation*}%
where for $0\leqslant t\leqslant T$
\begin{equation}  \label{expression V^}
\widehat{V}_{t}=X_{t}+\widehat{\theta }_{T}\Sigma _{t},
\end{equation}%
and $\widehat{\theta }_{T}$ is our continuous LSE for $\theta $ as given in (%
\ref{theta^fractional case}). We will need the following lemmas.

\begin{lemma}
\label{lemme cv ps} Assume $H \in\left(\frac12,1\right)$. Then, as $%
T\rightarrow\infty$
\begin{eqnarray}
&&\frac{1}{T}\int_{0}^{T}X_{t}^{2}dt\longrightarrow\eta^X,
\label{convergence S_T} \\
&&\frac{1}{T}\int_{0}^{T}\Sigma_{t}^{2}dt\longrightarrow\eta^\Sigma,
\label{convergence Sigma_T} \\
&&\frac{1}{T}\int_{0}^{T}\Sigma_{t} X_tdt\longrightarrow0
\label{convergence sigma_T X_T}
\end{eqnarray}
almost surely, where
\begin{equation*}
\eta^X=\frac{H\Gamma(2H)}{\rho^2-\theta^2}[\rho^{2-2H}-\theta^{2-2H}],
\end{equation*}
and
\begin{equation*}
\eta^{\Sigma}=\frac{H\Gamma(2H)}{\rho^2-\theta^2}[\theta^{-2H}-\rho^{-2H}].
\end{equation*}
\end{lemma}

\begin{proof}
From (\ref{OUFOU}) we can write
\begin{equation*}
d\left(
\begin{matrix}
X_{t} \\
\Sigma _{t}%
\end{matrix}%
\right) =A\left(
\begin{matrix}
X_{t} \\
\Sigma _{t}%
\end{matrix}%
\right) dt+d\left(
\begin{matrix}
B_{t}^{H} \\
0%
\end{matrix}%
\right)
\end{equation*}%
where $A=\left(
\begin{matrix}
\theta +\rho & -\theta \rho \\
1 & 0%
\end{matrix}%
\right) .$ The process $\left(
\begin{matrix}
X_{t} \\
\Sigma _{t}%
\end{matrix}%
\right) $ is geometrically ergodic because the largest eigenvalue of $A$ is
negative. Then to prove Lemma \ref{lemme cv ps}, using Birkhoff's ergodic
theorem (for instance see \cite{hairer05}), it is sufficient to study the
convergence of $\mathbf{E}[X_{t}^{2}]$, $\mathbf{E}[\Sigma _{t}^{2}]$ and $%
\mathbf{E}[\Sigma _{t}X_{t}]$ as $t\longrightarrow \infty $.\newline
For the convergence of $\mathbf{E}[X_{t}^{2}]$, (\ref{representationX})
leads to
\begin{equation*}
\mathbf{E}[X_{t}^{2}]=\left( \frac{\rho }{\rho -\theta }\right)^2 \mathbf{E}%
[(X_{t}^{\rho })^{2}]+\left( \frac{\theta }{\theta -\rho }\right) ^{2}%
\mathbf{E}[(X_{t}^{\theta })^{2}]-\frac{2\theta \rho }{(\theta -\rho )^{2}}%
\mathbf{E}[(X^{\theta })_{t}(X_{t}^{\rho })].
\end{equation*}%
Since
\begin{equation*}
\eta ^{X}=\left( \frac{\rho }{\rho -\theta }\right)^2 \lambda (\rho ,\rho
)+\left( \frac{\theta }{\theta -\rho }\right) ^{2}\lambda (\theta ,\theta )-%
\frac{2\theta \rho }{(\theta -\rho )^{2}}\lambda (\theta ,\rho )
\end{equation*}%
then by using 1) of Lemma \ref{1mainInequalities} we obtain
\begin{equation*}
\left\vert \eta ^{X}-\mathbf{E}[X_{t}^{2}]\right\vert \leqslant c(H,\theta
,\rho )e^{-t/2}.
\end{equation*}%
Thus we deduce the convergence (\ref{convergence S_T}). \newline
Using the same argument and the fact that
\begin{equation}
\Sigma _{t}=\frac{V_{t}-X_{t}}{\theta }=\frac{X_{t}^{\theta }-X_{t}^{\rho }}{%
\rho -\theta }  \label{representationSigma}
\end{equation}%
we deduce the convergence (\ref{convergence Sigma_T}). \newline
Finally, the convergence (\ref{convergence sigma_T X_T}) is satisfied by
using $\int_{0}^{T}\Sigma _{t}X_{t}dt=\frac{\Sigma _{T}}{2}$ and point 5) of
Lemma \ref{1mainInequalities}.
\end{proof}

\begin{lemma}
\label{convergence numerator theta}We have
\begin{eqnarray}
\frac{1}{T}\int_0^T X_t\delta X_t\longrightarrow -(\rho+\theta)\eta^X
\end{eqnarray}%
almost surely as $T\longrightarrow\infty$.
\end{lemma}

\begin{proof}
From (\ref{OUFOU}) and (\ref{link}) we can write
\begin{eqnarray*}
\int_{0}^{T}X_{t}\delta X_{t} &=&-\theta \int_{0}^{T}X_{t}^{2}dt-\rho
\int_{0}^{T}X_{t}V_{t}dt+\int_{0}^{T}X_{t}\delta B_{t}^{H} \\
&=&-\theta \int_{0}^{T}X_{t}^{2}dt-\theta \rho \int_{0}^{T}X_{t}\Sigma
_{t}dt-\rho \int_{0}^{T}X_{t}^{2}dt+\int_{0}^{T}X_{t}dB_{t}^{H} \\
&&\quad -\alpha _{H}\int_{0}^{T}\int_{0}^{t}D_{s}X_{t}(t-s)^{2H-2}dsdt
\end{eqnarray*}%
where $\alpha _{H}=2H(2H-1)$. Moreover,
\begin{eqnarray*}
\int_{0}^{T}X_{t}dB_{t}^{H} &=&\int_{0}^{T}X_{t}dX_{t}+(\theta +\rho
)\int_{0}^{T}X_{t}^{2}dt+\theta \rho \int_{0}^{T}X_{t}\Sigma _{t}dt \\
&=&\frac{X_{T}^{2}}{2}+(\theta +\rho )\int_{0}^{T}X_{t}^{2}dt+\theta \rho
\int_{0}^{T}X_{t}\Sigma _{t}dt.
\end{eqnarray*}%
Thus
\begin{equation*}
\int_{0}^{T}X_{t}\delta X_{t}=\frac{X_{T}^{2}}{2}-\alpha
_{H}\int_{0}^{T}\int_{0}^{t}D_{s}X_{t}(t-s)^{2H-2}dsdt.
\end{equation*}%
Since
\begin{equation*}
D_{s}X_{t}^{m}=e^{-m(t-s)}1_{[0,t]}(s)
\end{equation*}%
we deduce that
\begin{eqnarray}
\frac{1}{T}\int_{0}^{T}X_{t}\delta X_{t} &=&\frac{X_{T}^{2}}{2T}-\frac{%
\alpha _{H}}{T}\int_{0}^{T}\int_{0}^{t}\frac{1}{\rho -\theta }(\rho e^{-\rho
(t-s)}-\theta e^{-\theta (t-s)})(t-s)^{2H-2}dsdt  \notag \\
&=&\frac{X_{T}^{2}}{2T}-\frac{\alpha _{H}}{T}\int_{0}^{T}\int_{0}^{t}\frac{1%
}{\rho -\theta }\left( \rho e^{-\rho r}-\theta e^{-\theta r}\right)
r^{2H-2}drdt.  \label{decomposition of
int_X_delta_X}
\end{eqnarray}%
Thanks to l'Hôpital's rule, as $T\longrightarrow\infty$
\begin{eqnarray*}
\frac{\alpha _{H}}{T}\int_{0}^{T}\int_{0}^{t}\frac{1}{\rho -\theta }\left(
\rho e^{\rho r}-\theta e^{\theta r}\right) r^{2H-2}drdt &\longrightarrow &%
\frac{H\Gamma (2H)}{\rho -\theta }[\rho ^{2-2H}-\theta ^{2-2H}] \\
&=&\frac{(\rho +\theta )\eta ^{X}}{\alpha _{H}}.
\end{eqnarray*}%
Finally, combining this last convergence and point 5) of Lemma \ref%
{1mainInequalities}, the proof of Lemma \ref{convergence numerator theta} is
done.
\end{proof}

We now have all the elements to obtain our strong consistency result for $%
\widehat{\theta }_{T}$.

\begin{theorem}
\label{consistence of theta} We have
\begin{equation*}
\widehat{\theta }_{T}\longrightarrow \theta ^{\ast }
\end{equation*}%
almost surely as $T\longrightarrow \infty $, where $\theta ^{\ast }=\theta
+\rho $.
\end{theorem}

\begin{proof}
The proof follows directly from the convergence (\ref{convergence S_T}) and
Lemma \ref{convergence numerator theta}.
\end{proof}

The next lemmas are additional elements needed to prove the strong
consistency of $\widehat{\rho }_{T}$.

\begin{lemma}
\label{convergence L_T hate} We have
\begin{equation*}
\frac{\widehat{L}_{T}}{T}\longrightarrow \eta ^{X}+(\rho +\theta )^{2}\eta
^{\Sigma }
\end{equation*}%
almost surely as $T\longrightarrow \infty $.
\end{lemma}

\begin{proof}
The equation (\ref{expression V^}) ensures
\begin{equation*}
\widehat{L}_{T}=\int_{0}^{T}X_{t}^{2}dt+2\widehat{\theta }%
_{T}\int_{0}^{T}X_{t}\Sigma _{t}dt+\widehat{\theta }_{T}^{2}\int_{0}^{T}%
\Sigma _{t}^{2}dt,
\end{equation*}%
and the desired conclusion follows by using Lemma \ref{lemme cv ps} and
Theorem \ref{consistence of theta}.
\end{proof}


\begin{lemma}
\label{convergence numerator of rho}We have
\begin{equation*}
\frac{1}{T}\int_{0}^{T}\widehat{V}_{t}\delta \widehat{V}_{t}\longrightarrow
-\rho \theta (\rho +\theta )\eta ^{\Sigma }
\end{equation*}%
almost surely as $T\longrightarrow \infty $.
\end{lemma}

\begin{proof}
From (\ref{OUFOU}) and (\ref{expression V^}) we can write
\begin{eqnarray*}
\int_{0}^{T}\widehat{V}_{t}\delta \widehat{V}_{t} &=&\int_{0}^{T}X_{t}\delta
X_{t}+\widehat{\theta }_{T}\int_{0}^{T}X_{t}^{2}dt-\widehat{\theta }%
_{T}(\theta +\rho )\int_{0}^{T}\Sigma _{t}X_{t}dt-\rho \theta \widehat{%
\theta }_{T}\int_{0}^{T}\Sigma _{t}^{2}dt \\
&&+\int_{0}^{T}\Sigma _{t}\widehat{\theta }_{T}\delta B_{t}^{H}+\widehat{%
\theta }_{T}^{2}\int_{0}^{T}\Sigma _{t}X_{t}dt.
\end{eqnarray*}%
On the other hand
\begin{eqnarray*}
&&\int_{0}^{T}\widehat{\theta }_{T}\Sigma _{t}dB_{t}^{H}  \notag \\
&=&\int_{0}^{T}\widehat{\theta }_{T}\Sigma _{t}dV_{t}+\rho \widehat{\theta }%
_{T}\int_{0}^{T}\Sigma _{t}V_{t}dt  \notag \\
&=&-\theta ^{-1}\int_{0}^{T}\widehat{\theta }_{T}X_{t}dV_{t}+\theta ^{-1}%
\widehat{\theta }_{T}\int_{0}^{T}V_{t}dV_{t}+\rho \widehat{\theta }%
_{T}\int_{0}^{T}\Sigma _{t}X_{t}dt+\rho \theta \widehat{\theta }%
_{T}\int_{0}^{T}\Sigma _{t}^{2}dt  \notag \\
&=&-\theta ^{-1}\int_{0}^{T}\widehat{\theta }_{T}X_{t}dX_{t}-\widehat{\theta
}_{T}\int_{0}^{T}X_{t}^{2}dt+\theta ^{-1}\widehat{\theta }%
_{T}\int_{0}^{T}V_{t}dV_{t}+\rho \widehat{\theta }_{T}\int_{0}^{T}\Sigma
_{t}X_{t}dt+\rho \theta \widehat{\theta }_{T}\int_{0}^{T}\Sigma _{t}^{2}dt
\notag \\
&=&\frac{-1}{2\theta }X_{T}^{2}-\widehat{\theta }_{T}\int_{0}^{T}X_{t}^{2}dt+%
\frac{\widehat{\theta }_{T}}{2\theta }V_{T}^{2}+\rho \widehat{\theta }%
_{T}\int_{0}^{T}\Sigma _{t}X_{t}dt+\rho \theta \widehat{\theta }%
_{T}\int_{0}^{T}\Sigma _{t}^{2}dt.  \label{J^1}
\end{eqnarray*}%
Now, applying (\ref{link}), we obtain
\begin{eqnarray}
\int_{0}^{T}\widehat{V}_{t}\delta \widehat{V}_{t} &=&\int_{0}^{T}X_{t}\delta
X_{t}-\frac{1}{2\theta }X_{T}^{2}+\frac{\widehat{\theta }_{T}}{2\theta }%
V_{T}^{2}-\widehat{\theta }_{T}\theta \int_{0}^{T}\Sigma _{t}X_{t}dt+%
\widehat{\theta }_{T}^{2}\int_{0}^{T}\Sigma _{t}X_{t}dt  \notag \\
&&-\alpha _{H}\int_{0}^{T}\int_{0}^{t}D_{s}(\widehat{\theta }_{T}\Sigma
_{t})(t-s)^{2H-2}dsdt.  \label{^v d ^v}
\end{eqnarray}%
On the other hand
\begin{equation*}
D_{s}(\widehat{\theta }_{T}\Sigma _{t})=\Sigma _{t}D_{s}\widehat{\theta }%
_{T}+\widehat{\theta }_{T}D_{s}\Sigma _{t}.
\end{equation*}%
It follows from (\ref{decomposition of int_X_delta_X}) that
\begin{eqnarray}
\widehat{\theta }_{T} &=&\frac{\int_{0}^{T}X_{t}\delta X_{t}}{S_{T}}  \notag
\\
&=&\frac{\frac{1}{2}X_{T}^{2}-\alpha _{H}\int_{0}^{T}\int_{0}^{t}\frac{1}{%
\rho -\theta }\left( \rho e^{-\rho r}-\theta e^{-\theta r}\right)
r^{2H-2}drdt}{S_{T}}.  \notag
\end{eqnarray}%
Hence, for $s<T$
\begin{equation*}
D_{s}\widehat{\theta }_{T}=\frac{X_{T}D_{s}X_{T}-\widehat{\theta }%
_{T}D_{s}S_{T}}{S_{T}}.
\end{equation*}%
Thus
\begin{eqnarray*}
&&\alpha _{H}\int_{0}^{T}\int_{0}^{t}D_{s}(\widehat{\theta }_{T}\Sigma
_{t})(t-s)^{2H-2}dsdt \\
&=&\alpha _{H}\frac{X_{T}}{S_{T}}\int_{0}^{T}\int_{0}^{t}\Sigma
_{t}D_{s}X_{T}(t-s)^{2H-2}dsdt-\alpha _{H}\frac{\widehat{\theta }_{T}}{S_{T}}%
\int_{0}^{T}\int_{0}^{t}\Sigma _{t}D_{s}S_{T}(t-s)^{2H-2}dsdt \\
&&+\alpha _{H}\widehat{\theta }_{T}\int_{0}^{T}\int_{0}^{t}D_{s}\Sigma
_{t}(t-s)^{2H-2}dsdt \\
&:=&J_{1,T}-J_{2,T}+J_{3,T}
\end{eqnarray*}%
We shall prove that for every $\varepsilon>0$
\begin{eqnarray}
&&\frac{|J_{1,T}|}{T^{\varepsilon }}\longrightarrow 0,
\label{convergence J_1} \\
&&\frac{|J_{2,T}|}{T^{\varepsilon }}\longrightarrow 0,
\label{convergence
J_2} \\
&&\frac{J_{3,T}}{T}\longrightarrow \frac{\theta +\rho }{\rho -\theta }%
H\Gamma (2H)[(-\rho )^{1-2H}-(-\theta )^{1-2H}]  \label{convergence
J_3}
\end{eqnarray}%
almost surely as $T\rightarrow \infty $.\newline
We first estimate $J_{1,T}$. Clearly, (\ref{representationX}) implies
\begin{eqnarray*}
J_{1,T} &=&\alpha _{H}\frac{X_{T}}{S_{T}}\int_{0}^{T}\int_{0}^{t}\Sigma
_{t}D_{s}X_{T}(t-s)^{2H-2}dsdt \\
&=&\frac{\alpha _{H}}{\rho -\theta }\frac{X_{T}}{S_{T}}\int_{0}^{T}\Sigma
_{t}\int_{0}^{t}\left( \rho e^{-\rho (T-s)}-\theta e^{-\theta (T-s)}\right)
(t-s)^{2H-2}dsdt \\
&=&\frac{\alpha _{H}}{\rho -\theta }\frac{X_{T}}{S_{T}}\int_{0}^{T}\Sigma
_{t}\int_{0}^{t}\left( \rho e^{-\rho (T-t+x)}-\theta e^{-\theta
(T-t+x)}\right) x^{2H-2}dxdt.
\end{eqnarray*}%
The last equality comes from making the change of variable $x=t-s$. \newline
Hence
\begin{equation*}
\frac{|J_{1,T}|}{T^{\varepsilon }}\leqslant c(H,\theta,\rho)\frac{%
|X_{T}|/T^{\varepsilon }}{S_{T}/T}\frac{\sup_{t\in \lbrack 0,T]}|\Sigma _{t}|%
}{T^{\varepsilon}}.
\end{equation*}%
Using (\ref{convergence S_T}), (\ref{representationSigma}) and the point 5)
of Lemma \ref{1mainInequalities}, the convergence (\ref{convergence J_1}) is
obtained. \newline
Next we estimate $J_{2,T}$. By (\ref{representationX}) we have
\begin{eqnarray*}
J_{2,T} &=&\alpha _{H}\frac{\widehat{\theta }_{T}}{S_{T}}\int_{0}^{T}%
\int_{0}^{t}\Sigma _{t}D_{s}S_{T}(t-s)^{2H-2}dsdt \\
&=&2\alpha _{H}\frac{\widehat{\theta }_{T}}{S_{T}}\int_{0}^{T}\int_{0}^{t}%
\Sigma _{t}\int_{s}^{T}X_{u}D_{s}X_{u}(t-s)^{2H-2}dudsdt \\
&=&2\alpha _{H}\frac{\widehat{\theta }_{T}}{(\rho -\theta )S_{T}}%
\int_{0}^{T}\int_{0}^{t}\Sigma _{t}\int_{s}^{T}X_{u}\left( \rho e^{-\rho
(u-s)}-\theta e^{-\theta (u-s)}\right) (t-s)^{2H-2}dudsdt.
\end{eqnarray*}%
Then
\begin{eqnarray*}
\frac{\left\vert J_{2,T}\right\vert }{T^{\varepsilon}} &\leqslant
&c(H,\theta,\rho)\frac{\left\vert \widehat{\theta }_{T}\right\vert }{S_{T}}%
\frac{\sup_{t\in \lbrack 0,T]}|X_{t}|\sup_{t\in \lbrack 0,T]}|\Sigma _{t}|}{%
T^{\varepsilon}}\int_{0}^{T}\int_{0}^{t} e^{-\min(\theta,\rho) (T-s)}
(t-s)^{2H-2}dsdt \\
&\leqslant & c(H,\theta,\rho)\frac{\left\vert \widehat{\theta }%
_{T}\right\vert }{S_{T}/T}\frac{\sup_{t\in \lbrack 0,T]}|X_{t}|\sup_{t\in
\lbrack 0,T]}|\Sigma _{t}|}{T^{\varepsilon}} \\
&\longrightarrow &0
\end{eqnarray*}
almost surely as $T\longrightarrow \infty $. The last convergence comes from
(\ref{convergence S_T}), (\ref{representationSigma}), Theorem \ref%
{consistence of theta} and the point 5) of Lemma \ref{1mainInequalities}.
Thus, the convergence (\ref{convergence J_2}) is satisfied. \newline
Finally, we estimate $J_{3,T}$. Using (\ref{representationSigma}) and (\ref%
{representationX})
\begin{eqnarray*}
J_{3,T} &=&\alpha _{H}\widehat{\theta }_{T}\int_{0}^{T}\int_{0}^{t}D_{s}%
\Sigma _{t}(t-s)^{2H-2}dsdt \\
&=&\frac{\alpha _{H}\widehat{\theta }_{T}}{\rho -\theta }\int_{0}^{T}%
\int_{0}^{t}\left( e^{-\theta (t-s)}-e^{-\rho (t-s)}\right) (t-s)^{2H-2}dsdt.
\end{eqnarray*}%
By l'Hôpital rule we obtain
\begin{equation*}
\frac{J_{3,T}}{T}\longrightarrow \frac{\theta +\rho }{\rho -\theta }H\Gamma
(2H)[\theta ^{1-2H}-\rho ^{1-2H}]
\end{equation*}%
almost surely as $T\longrightarrow \infty $. \newline
Using the above estimations (\ref{convergence J_1}), (\ref{convergence J_2}%
), (\ref{convergence J_3}) together with (\ref{^v d ^v}), Lemma \ref{lemme
cv ps}, Theorem \ref{consistence of theta}, the point 5) of Lemma \ref%
{1mainInequalities} and Lemma \ref{convergence numerator theta} the desired
result is then obtained.
\end{proof}

\begin{theorem}
\label{consistence of rho} We have the almost sure convergence
\begin{equation*}
\widehat{\rho}_T\longrightarrow\rho^{*}
\end{equation*}
as $T\rightarrow \infty$, where
\begin{equation*}
\rho^{*} =\frac{\theta\rho(\theta+\rho)\eta^{\Sigma}}{\eta^{X}+(\theta+%
\rho)^2\eta^{\Sigma}}.
\end{equation*}
\end{theorem}

\begin{proof}
The proof is a straightforward consequence of Lemma \ref{convergence L_T
hate} and Lemma \ref{convergence numerator of rho}.
\end{proof}

Our approach to prove the asymptotic normality for both estimators $\widehat{%
\theta }_{T}$ and $\widehat{\rho }_{T}$ looks first at the normal
convergence of the $T$-indexed second-chaos sequence based on the kernel
which appears in the representation (\ref{Xm}) of $X$. Thereafter, thanks to
elementary stochastic calculus in the second chaos, these double stochastic
integrals will be identified in an expression for the leading terms in $%
\widehat{\theta }_{T}-\theta ^{\ast }$ in the proof of Theorem \ref%
{convergence in law theorem}. A similar technique, plus the use of the chain
rule of Young integrals and their relation to Skorohod integrals, is used to
find again that the leading terms in $\widehat{\rho }_{T}-\rho ^{\ast }$ are
also linear combinations of the same double integrals; the analysis of the
lower-order terms are less evident than for $\widehat{\theta }_{T}-\theta
^{\ast }$; the proof of Theorem \ref{convergence in law theorem} records all
the details.

We are ready to prove the asymptotic normality of $\left( \widehat{\theta }%
_{T},\widehat{\rho }_{T}\right) $.

\begin{theorem}
\label{convergence in law theorem} Assume that $H\in(\frac12,\frac34)$. Then
\begin{equation*}
\sqrt{T}\left( \widehat{\theta }_{T}-\theta ^{\ast },\widehat{\rho }%
_{T}-\rho ^{\ast }\right) \overset{law}{\longrightarrow }\mathcal{N}(0, ^tP\
\Gamma\ P )
\end{equation*}%
where the matrices $\Gamma $ and $P $ are defined respectively in (\ref%
{matrix Gamma}) and (\ref{matrix Sigma}).
\end{theorem}

\begin{proof}
We express $\widehat{\theta }_{T}-\theta ^{\ast }$ and $\widehat{\rho }%
_{T}-\rho ^{\ast }$ as linear combinations of the double stochastic
integrals identified in the previous theorem, plus lower-order terms.\newline
The case of $\widehat{\theta }_{T}-\theta ^{\ast }$ is rather
straightforward. It follows from (\ref{SIDE}) that
\begin{equation*}
\widehat{\theta }_{T}-\theta ^{\ast }=\frac{\rho \theta
\int_{0}^{T}X_{t}\Sigma _{t}dt-\int_{0}^{T}X_{t}\delta B_{t}}{S_{T}}.
\end{equation*}%
Since
\begin{equation*}
\int_{0}^{T}X_{t}\Sigma _{t}dt=\frac{1}{2}\Sigma _{T}^{2}
\end{equation*}%
and
\begin{eqnarray*}
\frac{1}{\sqrt{T}}\int_{0}^{T}X_{t}\delta B_{t}^{H} &=&\frac{1}{(\rho
-\theta )\sqrt{T}}\int_{0}^{T}\int_{0}^{t}(\rho e^{-\rho (t-s)}-\theta
e^{-\theta (t-s)})\delta B_{s}^{H}\delta B_{t}^{H} \\
&=&\frac{1}{(\rho -\theta )\sqrt{T}}\left( \rho I_{2}(f_{T}^{\rho })-\theta
I_{2}(f_{T}^{\theta })\right)
\end{eqnarray*}%
we can write
\begin{equation}
\sqrt{T}\left( \widehat{\theta }_{T}-\theta ^{\ast }\right) =\frac{\frac{1}{%
(\rho -\theta )\sqrt{T}}\left( \theta I_{2}(f_{T}^{\theta })-\rho
I_{2}(f_{T}^{\rho })\right) }{S_{T}/T}+R_{T}^{\theta }.
\label{expression theta-theta*}
\end{equation}%
where
\begin{equation}
R_{T}^{\theta }:=\frac{\rho \theta }{2}\frac{\Sigma _{T}^{2}/\sqrt{T}}{%
S_{T}/T}\longrightarrow 0  \label{Rthetaneg}
\end{equation}%
almost surely as $T\longrightarrow \infty $. \newline
For $\widehat{\rho }_{t}-\rho ^{\ast }$, the situation is significantly more
complex. We have for every $0\leqslant t\leqslant T$
\begin{eqnarray*}
\widehat{V}_{t} &=&X_{t}+\widehat{\theta }_{T}\Sigma _{t}=V_{t}+(\widehat{%
\theta }_{T}-\theta )\Sigma _{t}=V_{t}+(\widehat{\theta }_{T}-\theta ^{\ast
})\Sigma _{t}+\rho \Sigma _{t} \\
&=&V_{t}-\frac{\rho }{\theta }(X_{t}-V_{t})-\frac{1}{\theta }(\widehat{%
\theta }_{T}-\theta ^{\ast })(X_{t}-V_{t}) \\
&=&\frac{\theta ^{\ast }}{\theta }V_{t}-\frac{\rho }{\theta }X_{t}-\frac{1}{%
\theta }(\widehat{\theta }_{T}-\theta ^{\ast })(X_{t}-V_{t})
\end{eqnarray*}%
which leads to
\begin{equation*}
\widehat{L}_{T}=\int_{0}^{T}\widehat{V}_{t}^{2}dt=I_{T}+(\widehat{\theta }%
_{T}-\theta ^{\ast })(J_{T}+(\widehat{\theta }_{T}-\theta ^{\ast })K_{T})
\end{equation*}%
where
\begin{eqnarray*}
I_{T} &=&\frac{1}{\theta ^{2}}(\rho ^{2}S_{T}+(\theta ^{\ast
})^{2}L_{T}-2\theta ^{\ast }\rho P_{T}), \\
J_{T} &=&\frac{1}{\theta ^{2}}(2\rho S_{T}+2\theta ^{\ast }L_{T}-2(\theta
+2\rho )P_{T}), \\
K_{T} &=&\frac{1}{\theta ^{2}}(S_{T}+L_{T}-2P_{T}).
\end{eqnarray*}%
Thus,
\begin{equation*}
\widehat{L}_{T}(\widehat{\rho }_{T}-\rho ^{\ast })=I_{T}^{V}+(\widehat{%
\theta }_{T}-\theta ^{\ast })(J_{T}^{V}+(\widehat{\theta }_{T}-\theta ^{\ast
})K_{T}^{V})
\end{equation*}%
where
\begin{eqnarray*}
I_{T}^{V} &=&-\int_{0}^{T}\widehat{V}_{t}\delta \widehat{V}_{t}-\rho ^{\ast
}I_{T}=-\frac{\widehat{V}_{T}^{2}}{2}+\alpha
_{H}\int_{0}^{T}\int_{0}^{t}D_{s}\widehat{V}_{t}(t-s)^{2H-2}dsdt-\rho ^{\ast
}I_{T}, \\
J_{T}^{V} &=&-\rho ^{\ast }J_{T},\mbox{ and }K_{T}^{V}=-\rho ^{\ast }K_{T}.
\end{eqnarray*}%
On the other hand, using the formula (\ref{chain rule}) we obtain
\begin{equation*}
\left\{
\begin{array}{lcl}
S_{T}=-\frac{X_{T}^{2}}{2\theta ^{\ast }}+\frac{1}{\theta ^{\ast }}%
\int_{0}^{T}X_{s}dB_{s}^{H}-\frac{\rho \theta }{2\theta ^{\ast }}\Sigma
_{T}^{2}, &  &  \\
P_{T}=-\frac{X_{T}V_{T}}{\theta ^{\ast }}+\frac{1}{\theta ^{\ast }}%
\int_{0}^{T}X_{s}dB_{s}^{H}+\frac{V_{T}^{2}}{2\theta ^{\ast }}, &  &  \\
L_{T}=-\frac{V_{T}^{2}}{2\rho }+\frac{1}{\rho }\int_{0}^{T}V_{s}dB_{s}^{H}.
&  &
\end{array}%
\right.
\end{equation*}%
Furthermore, using the relation between Young and Skorohod integrals,
\begin{equation}
\left\{
\begin{array}{lcl}
S_{T}=-\frac{X_{T}^{2}}{2\theta ^{\ast }}+\frac{1}{\theta ^{\ast }}%
\int_{0}^{T}X_{s}\delta B_{s}^{H}+\frac{\alpha _{H}}{\theta ^{\ast }}%
\int_{0}^{T}\int_{0}^{t}D_{s}X_{t}(t-s)^{2H-2}dsdt-\frac{\rho \theta }{%
2\theta ^{\ast }}\Sigma _{T}^{2}, &  &  \\
P_{T}=-\frac{X_{T}V_{T}}{\theta ^{\ast }}+\frac{1}{\theta ^{\ast }}%
\int_{0}^{T}X_{s}\delta B_{s}^{H}+\frac{\alpha _{H}}{\theta ^{\ast }}%
\int_{0}^{T}\int_{0}^{t}D_{s}X_{t}(t-s)^{2H-2}dsdt+\frac{V_{T}^{2}}{2\theta
^{\ast }}, &  &  \\
L_{T}=-\frac{V_{T}^{2}}{2\rho }+\frac{1}{\rho }\int_{0}^{T}V_{s}\delta
B_{s}^{H}+\frac{\alpha _{H}}{\rho }\int_{0}^{T}%
\int_{0}^{t}D_{s}V_{t}(t-s)^{2H-2}dsdt. &  &
\end{array}%
\right.
\end{equation}%
Setting
\begin{equation*}
\lambda _{T}:=\alpha _{H}\int_{0}^{T}\int_{0}^{t}e^{-(t-s)}(t-s)^{2H-2}dsdt
\end{equation*}%
we can write
\begin{eqnarray*}
\lambda _{T}^{X}:= &&\alpha
_{H}\int_{0}^{T}\int_{0}^{t}D_{s}X_{t}(t-s)^{2H-2}dsdt=\frac{1}{\rho -\theta
}\left( \rho ^{2-2H}-\theta ^{2-2H}\right) \lambda _{T} \\
\lambda _{T}^{V}:= &&\alpha
_{H}\int_{0}^{T}\int_{0}^{t}D_{s}V_{t}(t-s)^{2H-2}dsdt=\rho ^{1-2H}\lambda
_{T} \\
\lambda _{T}^{\Sigma }:= &&\alpha _{H}\int_{0}^{T}\int_{0}^{t}D_{s}\Sigma
_{t}(t-s)^{2H-2}dsdt=\frac{1}{\theta }(\lambda _{T}^{V}-\lambda _{T}^{X})=%
\frac{1}{\rho -\theta }\left( \theta ^{1-2H}-\rho ^{1-2H}\right) \lambda _{T}
\\
\lambda _{T}^{\widehat{V}}:= &&\alpha _{H}\int_{0}^{T}\int_{0}^{t}D_{s}%
\widehat{V}_{t}(t-s)^{2H-2}dsdt=\lambda _{T}^{X}+\widehat{\theta }%
_{T}\lambda _{T}^{\Sigma }+J_{1,T}-J_{2,T} \\
&=&\lambda _{T}^{X}+\theta ^{\ast }\lambda _{T}^{\Sigma }+(\widehat{\theta }%
_{T}-\theta ^{\ast })\lambda _{T}^{\Sigma }+J_{1,T}-J_{2,T}.
\end{eqnarray*}%
The last equality comes from the fact that $D_{s}\widehat{V}_{t}=D_{s}X_{t}+%
\widehat{\theta }_{T}D_{s}\Sigma _{t}+\Sigma _{t}D_{s}\widehat{\theta }_{T}$%
. \newline
Since
\begin{equation*}
\lambda _{T}^{X}+\theta ^{\ast }\lambda _{T}^{\Sigma }=\frac{\rho ^{\ast }}{%
\theta ^{2}}\left( -\frac{\rho ^{2}}{\theta ^{\ast }}\lambda _{T}^{X}-\frac{%
(\theta ^{\ast })^{2}}{\rho }\lambda _{T}^{V}+2\rho \lambda _{T}^{X}\right)
\end{equation*}%
we can write
\begin{eqnarray*}
I_{T}^{V} &=&-\frac{\widehat{V}_{T}^{2}}{2}+\lambda _{T}^{\widehat{V}}-\rho
^{\ast }I_{T} \\
&=&(\widehat{\theta }_{T}-\theta ^{\ast })\lambda _{T}^{\Sigma }-\frac{\rho
^{\ast }}{\theta ^{2}}\left[ (-2\rho +\frac{\rho ^{2}}{\theta ^{\ast }}%
)\int_{0}^{T}X_{s}\delta B_{s}^{H}+\frac{(\theta ^{\ast })^{2}}{\rho }%
\int_{0}^{T}V_{s}\delta B_{s}^{H}\right] +R_{T}
\end{eqnarray*}%
where
\begin{equation*}
R_{T}=\frac{-\widehat{V}_{T}^{2}}{2}+J_{1,T}-J_{2,T}-\frac{\rho ^{\ast }}{%
\theta ^{2}}\left[ \rho ^{2}(-\frac{X_{T}^{2}}{2\theta ^{\ast }}-\frac{\rho
\theta }{2\theta ^{\ast }}\Sigma _{T}^{2}-(\theta ^{\ast })^{2}\frac{%
V_{T}^{2}}{2\rho }-2\theta ^{\ast }\rho (\frac{-X_{T}V_{T}}{\theta ^{\ast }}+%
\frac{V_{T}^{2}}{2\theta ^{\ast }})\right] .
\end{equation*}%
Combining previous estimations we obtain
\begin{equation*}
\widehat{L}_{T}(\widehat{\rho }_{T}-\rho ^{\ast })=c_{T}^{\rho
}I_{2}(f_{T}^{\rho })+c_{T}^{\theta }I_{2}(f_{T}^{\theta })+\frac{%
R_{T}^{\theta }}{\sqrt{T}}\left( \lambda _{T}^{\Sigma }+J_{T}^{V}+(\widehat{%
\theta }_{T}-\theta ^{\ast })K_{T}^{V}\right) +R_{T}
\end{equation*}%
where
\begin{eqnarray}
c_{T}^{\rho } &=&\frac{-\rho \lambda _{T}^{\Sigma }}{(\rho -\theta )S_{T}}-%
\frac{\rho \rho ^{\ast }(-2\rho +\frac{\rho ^{2}}{\theta ^{\ast }})}{\theta
^{2}(\rho -\theta )}-\frac{\rho ^{\ast }(\theta ^{\ast })^{2}}{\rho \theta
^{2}}-\frac{\rho J_{T}^{V}}{(\rho -\theta )S_{T}}-\frac{\rho (\widehat{%
\theta }_{T}-\theta ^{\ast })K_{T}^{V}}{(\rho -\theta )S_{T}}  \notag \\
&\longrightarrow &c^{\rho }=\frac{-\rho \lambda ^{\Sigma }}{(\rho -\theta
)\eta ^{X}}-\frac{\rho \rho ^{\ast }(-2\rho +\frac{\rho ^{2}}{\theta ^{\ast }%
})}{\theta ^{2}(\rho -\theta )}-\frac{\rho ^{\ast }(\theta ^{\ast })^{2}}{%
\rho \theta ^{2}}-\frac{\rho \lambda ^{J}}{(\rho -\theta )\eta ^{X}}
\end{eqnarray}%
almost surely as $T\longrightarrow \infty $, and
\begin{eqnarray}
c_{T}^{\theta } &=&\frac{\theta \lambda _{T}^{\Sigma }}{(\rho -\theta )S_{T}}%
+\frac{\theta \rho ^{\ast }(-2\rho +\frac{\rho ^{2}}{\theta ^{\ast }})}{%
\theta ^{2}(\rho -\theta )}-\frac{\rho ^{\ast }(\theta ^{\ast })^{2}}{\rho
\theta ^{2}}+\frac{\theta J_{T}^{V}}{(\rho -\theta )S_{T}}+\frac{\theta (%
\widehat{\theta }_{T}-\theta ^{\ast })K_{T}^{V}}{(\rho -\theta )S_{T}}
\notag \\
&\longrightarrow &c^{\theta }=\frac{\theta \lambda ^{\Sigma }}{(\rho -\theta
)\eta ^{X}}+\frac{\theta \rho ^{\ast }(-2\rho +\frac{\rho ^{2}}{\theta
^{\ast }})}{\theta ^{2}(\rho -\theta )}-\frac{\rho ^{\ast }(\theta ^{\ast
})^{2}}{\rho \theta ^{2}}+\frac{\theta \lambda ^{J}}{(\rho -\theta )\eta ^{X}%
}
\end{eqnarray}%
almost surely as $T\longrightarrow \infty $. These last two convergences
come from the fact that
\begin{equation*}
\lambda _{T}^{\Sigma }/T\longrightarrow \lambda ^{\Sigma }=\frac{2H\Gamma
(2H)}{\rho -\theta }\left[ \theta ^{1-2H}-\rho ^{1-2H}\right]
\end{equation*}%
and
\begin{equation*}
J_{T}^{V}/T\longrightarrow \lambda ^{J}=\frac{2}{\theta ^{2}}\left[ \theta
^{\ast }\eta ^{X}-H\Gamma (2H)\theta ^{\ast }\rho ^{-2H}\right]
\end{equation*}%
almost surely as $T\longrightarrow \infty $, because $\lambda
_{T}/T\longrightarrow 2H\Gamma (2H)$ as $T\longrightarrow \infty $. \newline
Thus
\begin{equation}
\sqrt{T}(\widehat{\rho }_{T}-\rho ^{\ast })=\frac{\frac{1}{\sqrt{T}}\left(
c_{T}^{\rho }I_{2}(f_{T}^{\rho })+c_{T}^{\theta }I_{2}(f_{T}^{\theta
})\right) }{\frac{\widehat{L}_{T}}{T}}+R_{T}^{\rho }
\label{expression rho-rho*}
\end{equation}%
where as $T\longrightarrow \infty $
\begin{equation}
R_{T}^{\rho }=R_{T}^{\theta }\frac{\left( \lambda _{T}^{\Sigma }+J_{T}^{V}+(%
\widehat{\theta }_{T}-\theta ^{\ast })K_{T}^{V}\right) }{\widehat{L}_{T}}+%
\frac{R_{T}/\sqrt{T}}{\frac{\widehat{L}_{T}}{T}}\longrightarrow 0
\label{Rrhoneg}
\end{equation}%
almost surely. \newline
Finally, with the expressions (\ref{expression theta-theta*}) and (\ref%
{expression rho-rho*}) on hand, and the almost-sure negligibility of their
corresponding lower-order terms as proved in (\ref{Rthetaneg}) and (\ref%
{Rrhoneg}), we get%
\begin{equation*}
\sqrt{T}\left( \widehat{\theta }_{T}-\theta ^{\ast },\widehat{\rho }%
_{T}-\rho ^{\ast }\right) =\frac{1}{\sqrt{T}}\left( I_{2}(f_{T}^{\theta
}),I_{2}(f_{T}^{\rho })\right) \left(
\begin{matrix}
\frac{\theta }{\rho -\theta }\frac{T}{S_{T}} & c_{T}^{\theta }\frac{T}{%
\widehat{L}_{T}} \\
\frac{\rho }{\theta -\rho }\frac{T}{S_{T}} & c_{T}^{\rho }\frac{T}{\widehat{L%
}_{T}}%
\end{matrix}%
\right) +\left( R_{T}^{\theta },R_{T}^{\rho }\right)
\end{equation*}%
where as $T\longrightarrow \infty $
\begin{equation}
\left(
\begin{matrix}
\frac{\theta }{\rho -\theta }\frac{T}{S_{T}} & c_{T}^{\theta }\frac{T}{%
\widehat{L}_{T}} \\
\frac{\rho }{\theta -\rho }\frac{T}{S_{T}} & c_{T}^{\rho }\frac{T}{\widehat{L%
}_{T}}%
\end{matrix}%
\right) \longrightarrow P :=\left(
\begin{matrix}
\frac{\theta }{\rho -\theta }\frac{1}{\eta ^{X}} & \frac{c^{\theta }}{\eta ^{%
\widehat{L}}} \\
\frac{\rho }{\theta -\rho }\frac{1}{\eta ^{X}} & \frac{c^{\rho }}{\eta ^{%
\widehat{L}}}%
\end{matrix}%
\right)  \label{matrix Sigma}
\end{equation}%
almost surely. Now, applying Slutsky's lemma and Theorem \ref{key asymptotic
normality} combined with the above convergences, the proof is complete.
\end{proof}

\section{Discrete observation}

Assume that the process $X$ is observed equidistantly in time with the step
size $\Delta _{n}$: $t_{i}=i\Delta _{n},i=0,\ldots ,n$, and $T_{n}=n\Delta
_{n}$ denotes the length of the `observation window'. The goal of this
section is to construct two estimators $\check{\theta}_{n}$ and $\check{\rho}%
_{n}$ of $\theta $ and $\rho $ respectively based on the sampling data $%
X_{t_{i}},i=0,\ldots ,n$, and study their strong consistency and asymptotic
normality. We also want to define estimators in such a way that consistency
and normality results proved in Section 3 for the continuous-data estimators
$\left( \widehat{\theta }_{T},\widehat{\rho }_{T}\right) $ can be used to
good effect in the discrete case. The basic strategy for this is therefore
to look for ways of discretizing the MLE\ studied in Section 3. It turns out
that the most efficient way of implementing this strategy is to define
several intermediate estimators, starting with ones where only the
denominators in $\left( \widehat{\theta }_{T},\widehat{\rho }_{T}\right) $
are discretized, and then using an algebraic asymptotic interpretation of
the numerators to avoid a direct discretization of the corresponding Young
or Skorohod integrals. This method allows a rather direct use of the
asymptotic normality Theorem \ref{convergence in law theorem} in Section 3,
while for the strong consistency results, some of the almost-sure
convergences proved in Section 3 are used directly, and additional ones are
newly established early on in this section. See Section \ref{Summary} for
other details about the heuristics which explain the choices made below in
this Section.

For any given process $Z$, define%
\begin{equation*}
Q_{n}(Z):=\frac{1}{n}\sum_{i=1}^{n}(Z_{t_{i-1}})^{2}
\end{equation*}%
The following well-known direct consequence of the Borel-Cantelli Lemma (see
e.g. \cite{KN}), will allows us to turn convergence rates in the $p$-th mean
into pathwise convergence rates. This is particularly efficient when working
with sequences in Wiener chaos.

\begin{lemma}
\label{pathwise convergence} Let $\gamma >0$ and $p_{0}\in \mathbb{N}$.
Moreover let $(Z_{n})_{n\in \mathbb{N}}$ be a sequence of random variables.
If for every $p\geq p_{0}$ there exists a constant $c_{p}>0$ such that for
all $n\in \mathbb{N}$,
\begin{equation*}
(E|Z_{n}|^{p})^{1/p}\leqslant c_{p}\cdot n^{-\gamma },
\end{equation*}%
then for all $\varepsilon >0$ there exists a random variable $\eta
_{\varepsilon }$ such that
\begin{equation*}
|Z_{n}|\leqslant \eta _{\varepsilon }\cdot n^{-\gamma +\varepsilon }\quad %
\mbox{almost surely}
\end{equation*}%
for all $n\in \mathbb{N}$. Moreover, $E|\eta _{\varepsilon }|^{p}<\infty $
for all $p\geq 1$.
\end{lemma}

As before we assume that $\Delta _{n}=t_{k+1}-t_{k}$ is a function of $n$
only. Of some importance, particularly for the purpose of proving normal
convergence theorems, is the case $n^{-\alpha }$ with a given $\alpha \in {\
\mathbb{R}}$. The case $\alpha >0$ implies that the observation frequency
must increase even as the horizon itself also increases. The case of $\alpha
=0$ is of special importance because it corresponds to a setup where the
observation frequency is fixed ( $\Delta _{n}=1$, no in-fill asymptotics,
only increasing horizon), which may be desirable in some applications. We
will see that for some almost-sure convergence results, we may even take a
time step $\Delta _{n}$ which grows with $n$. In other words, this allows
for very sparse observations. We will also see that most almost-sure results
are valid for the entire range $H\in (\frac{1}{2},1)$, while normal
convergence results require $H\in (\frac{1}{2},\frac{3}{4})$. We begin by
recording and proving some important technical estimates.

\begin{lemma}
\label{cv Q(X)} Define $\delta _{n}(X):=\sqrt{T_{n}}\left( Q_{n}(X)-\frac{%
S_{T_{n}}}{T_{n}}\right) $. Then
\begin{equation}
E\left[ \delta _{n}^{2}(X)\right] \leqslant c(H,\theta ,\rho )\min \left(
n\Delta _{n}^{2H+1},\frac{1}{n\Delta _{n}}+\Delta _{n}^{H+1}+\Delta
_{n}^{4H-3}\sum_{j=1}^{n}j^{4H-4}\right) .  \label{cv Q(X)-S}
\end{equation}%
In particular, if $\Delta _{n}\leqslant n^{\alpha }\ $for some $\alpha \in
(-\infty ,1/H)$, then
\begin{equation}
Q_{n}(X)\longrightarrow \eta ^{X}  \label{convergence Q(X)}
\end{equation}%
almost surely as $n\rightarrow \infty $.
\end{lemma}

\begin{proof}
The points 3) and 6) of Lemma \ref{1mainInequalities} lead to
\begin{eqnarray*}
E\left[ \delta _{n}^{2}(X)\right] &=&\frac{1}{T_{n}}\sum_{i,j=1}^{n}%
\int_{t_{i-1}}^{t_{i}}\int_{t_{j-1}}^{t_{j}}E\left[ \left(
X_{t}^{2}-X_{t_{i-1}}^{2}\right) \left( X_{s}^{2}-X_{t_{j-1}}^{2}\right) %
\right] dsdt \\
&\leqslant&c(H,\theta,\rho)n\Delta _{n}^{2H+1}.
\end{eqnarray*}%
On the other hand, we have
\begin{eqnarray*}
&&E\left[ \delta _{n}^{2}(X)\right] \\
&=&\frac{1}{T_{n}}\sum_{i,j=1}^{n}\int_{t_{i-1}}^{t_{i}}%
\int_{t_{j-1}}^{t_{j}}E\left( X_{t}^{2}-X_{t_{i-1}}^{2}\right) E\left(
X_{s}^{2}-X_{t_{j-1}}^{2}\right) dsdt \\
&&+\frac{1}{T_{n}}\sum_{i,j=1}^{n}\int_{t_{i-1}}^{t_{i}}%
\int_{t_{j-1}}^{t_{j}}E\left[ \left( X_{t}-X_{t_{i-1}}\right) \left(
X_{s}-X_{t_{j-1}}\right) \right] E\left[ \left( X_{t}+X_{t_{i-1}}\right)
\left( X_{s}+X_{t_{j-1}}\right) \right] dsdt \\
&&+\frac{1}{T_{n}}\sum_{i,j=1}^{n}\int_{t_{i-1}}^{t_{i}}%
\int_{t_{j-1}}^{t_{j}}E\left[ \left( X_{t}-X_{t_{i-1}}\right) \left(
X_{s}+X_{t_{j-1}}\right) \right] E\left[ \left( X_{t}+X_{t_{i-1}}\right)
\left( X_{s}-X_{t_{j-1}}\right) \right] dsdt \\
:= &&\frac{1}{T_{n}}\sum_{i,j=1}^{n}\left(
D_{1}(i,j)+D_{2}(i,j)+D_{3}(i,j)\right) .
\end{eqnarray*}%
By using the points 2), 4) and 6) of Lemma \ref{1mainInequalities} we obtain
\begin{eqnarray*}
\frac{1}{T_{n}}\sum_{i,j=1}^{n}D_{1}(i,j) &=&\frac{1}{T_{n}}\left[
\sum_{i=1}^{n}\int_{t_{i-1}}^{t_{i}}E\left( X_{t}^{2}-X_{t_{i-1}}^{2}\right)
dt\right] ^{2} \\
&\leqslant &\frac{c(H,\theta ,\rho )}{T_{n}}\left[ \Delta
_{n}\sum_{i=1}^{n}e^{-t_{i-1}/2}\right] ^{2} \\
&\leqslant &\frac{c(H,\theta ,\rho )}{T_{n}}\left[ \frac{\Delta _{n}}{%
1-e^{-\Delta _{n}/2}}\right] ^{2},
\end{eqnarray*}%
\begin{eqnarray*}
&&\frac{1}{T_{n}}\sum_{i=1}^{n}\left( D_{2}(i,i)+D_{3}(i,i)\right) \\
&=&\frac{1}{T_{n}}\sum_{i=1}^{n}\left(
\int_{t_{i-1}}^{t_{i}}\int_{t_{i-1}}^{t_{i}}\left[
2(E(X_{t}X_{s}))^{2}+2(E(X_{t_{i-1}}^{2}))^{2}-4(E(X_{t_{i-1}}X_{s}))^{2}%
\right] dsdt\right) \\
&\leqslant &c(H,\theta ,\rho )\Delta _{n}^{H+1},
\end{eqnarray*}%
and
\begin{eqnarray*}
&&\frac{1}{T_{n}}\sum_{i\neq j=1}^{n}\left( D_{2}(i,j)+D_{3}(i,j)\right) \\
&=&\frac{2}{T_{n}}\sum_{i\neq j=1}^{n}\left(
\int_{t_{i-1}}^{t_{i}}\int_{t_{i-1}}^{t_{i}}\left[
(E(X_{t}X_{s}))^{2}-(E(X_{t_{i-1}}X_{s}))^{2}-(E(X_{t_{i-1}}X_{t}))^{2}+(E(X_{t_{i-1}}X_{t_{j-1}}))^{2}%
\right] dsdt\right) \\
&\leqslant &\frac{c(H,\theta ,\rho )}{T_{n}}\sum_{i<j=1}^{n}\Delta
_{n}^{4H-2}|j-i-1|^{4H-4} \\
&\leqslant &c(H,\theta ,\rho )\Delta _{n}^{4H-3}\sum_{j=1}^{n}j^{4H-4}.
\end{eqnarray*}%
Thus (\ref{cv Q(X)-S}) is obtained. Now, using (\ref{hypercontractivity}),
Lemma \ref{pathwise convergence} and (\ref{convergence S_T}) we will be able
to assert the convergence (\ref{convergence Q(X)}), and thus the entire
lemma, as soon as we can show that the right-hand side of (\ref{cv Q(X)-S})
divided by $T_{n}$ converges to $0$ as fast as some negative power of $n$.
Thus we only need to show that there exists $\varepsilon >0$, such that as $%
n\rightarrow \infty $%
\begin{equation*}
q_{n}:=\min \left( \Delta _{n}^{2H},\frac{1}{\left( n\Delta _{n}\right) ^{2}}%
+\frac{\Delta _{n}^{H}}{n}+\frac{1}{n}\Delta
_{n}^{4H-4}\sum_{j=1}^{n}j^{4H-4}\right) \leqslant n^{-\varepsilon }.
\end{equation*}

Let us concentrate first on the second part of the minimum defining $q_{n}$.
This is the sum of the three terms $\left( n\Delta _{n}\right) ^{-2}$, $%
\Delta _{n}^{H}/n$, and $n^{-1}\Delta _{n}^{4H-4}\sum_{j=1}^{n}j^{4H-4}$.
The first of these three terms will tend to $0$ like a negative power of $n$
as soon as there exists $\varepsilon _{1}>0$ such that $\Delta _{n}\geq
n^{-1+\varepsilon _{1}}.$ The second term will tend to $0$ like a negative
power of $n$ as soon as there exists $\varepsilon _{2}>0$ such that $\Delta
_{n}\leqslant n^{1/H-\varepsilon _{2}}$. For the third term, we must
separate the case $H<3/4$ from the case $H\geq 3/4$. When $H<3/4$, the
series $\sum_{j=1}^{n}j^{4H-4}$ is bounded, so the last term in the second
part of the min in $q_{n}$ will tend to $0$ like a negative power of $n$ as
soon as there exists $\varepsilon _{3}>0$ such that $\Delta _{n}\geq
n^{-1/(4-4H)+\varepsilon _{3}}$. When \thinspace $H>3/4$, the series is
bounded above by a constant times $n^{4H-3}$, yielding a contribution of $%
\left( \Delta _{n}/n\right) ^{4H-4}$; so the last term in the second part of
the min in $q_{n}$ will tend to $0$ like a negative power of $n$ as soon as
there exists $\varepsilon _{4}>0$ such that $\Delta _{n}\geq
n^{-1+\varepsilon _{4}}$. The case $H=3/4$ is done in the same fashion, with
the same conclusion as when $H>3/4$. Thus we have proved that for each fixed
$n$, if there exist $\varepsilon _{1},\varepsilon _{2},\varepsilon _{3}>0$
such that
\begin{equation}
\max \left( n^{-1+\varepsilon _{1}},n^{-\frac{1}{4-4H}+\varepsilon
_{3}}\right) \leqslant \Delta _{n}\leqslant n^{1/H-\varepsilon _{2}}
\label{Deltan first condition}
\end{equation}%
then for some $\varepsilon >0$,%
\begin{equation*}
q_{n}\leqslant n^{-\varepsilon }.
\end{equation*}
On the other hand notice that for every $H\in \left( 1/2,1\right) $, there
exist $\varepsilon _{1},\varepsilon _{3}>0$ such that
\begin{equation}
\max \left( n^{-1+\varepsilon _{1}},n^{-\frac{1}{4-4H}+\varepsilon
_{3}}\right) \leqslant n^{-1/2}.  \label{dicho}
\end{equation}%
Thus for each fixed $n$, if we have
\begin{equation}
\Delta _{n}\leqslant n^{-1/2},  \label{Deltan second condition}
\end{equation}%
using the first part of the $\min $ in the definition of $q_{n}$, we get%
\begin{equation*}
q_{n}\leqslant n^{-\varepsilon }
\end{equation*}%
with $\varepsilon =H$ . To conclude, by (\ref{dicho}), for each fixed $n$,
we are either in the case (\ref{Deltan first condition}) or (\ref{Deltan
second condition}), so that $q_{n}\leqslant n^{-\varepsilon }$ in all cases
as soon as $\Delta _{n}\leqslant n^{1/H-\varepsilon _{2}}$ for some $%
\varepsilon _{2}>0$. The proof of the lemma is complete.
\end{proof}

\begin{lemma}
\label{cv Q(Sigma)} Define $\delta _{n}(\Sigma ):=\sqrt{T_{n}}\left(
Q_{n}(\Sigma )-\frac{1}{T_{n}}\int_{0}^{T_{n}}\Sigma _{t}^{2}dt\right) $.
Then
\begin{equation}
E\left[ \delta _{n}^{2}(\Sigma )\right] \leqslant c(H,\theta ,\rho )\min
\left( n\Delta _{n}^{2H+1},\frac{1}{n\Delta _{n}}+\Delta _{n}^{H+1}+\Delta
_{n}^{4H-3}\sum_{j=1}^{n}j^{4H-4}\right).  \label{bound Q(Sigma)-intSigma}
\end{equation}%
In particular, if $\Delta _{n}\leqslant n^{\alpha }\ $for some $\alpha \in
(-\infty ,1/H)$, then
\begin{equation}
Q_{n}(\Sigma )\longrightarrow \eta ^{\Sigma }  \label{convergence Q(Sigma)}
\end{equation}%
almost surely as $n\rightarrow \infty $.\newline
On the other hand if $n^{1+\alpha }\Delta _{n}^{H+1}\rightarrow 0$ for some $%
\alpha >0$,
\begin{equation}
\left\vert Q_{n}(\widehat{\Sigma })-Q_{n}(\Sigma )\right\vert
\longrightarrow 0  \label{convergence Q(SigmaHat)}
\end{equation}%
almost surely as $n\rightarrow \infty $, where
\begin{equation*}
\widehat{\Sigma }_{t_{i}}^{n}=\Delta _{n}\sum_{j=1}^{i}X_{t_{j-1}}.
\end{equation*}%
In addition, if $n^{3}\Delta _{n}^{2H+3}\rightarrow 0$,
\begin{equation}
\sqrt{T_{n}}\left\vert Q_{n}(\widehat{\Sigma })-Q_{n}(\Sigma )\right\vert
\longrightarrow 0  \label{convergence Q(Sigma)-Q(SigmaHat)}
\end{equation}%
in $L^{2}(\Omega )$ as $n\rightarrow \infty $.
\end{lemma}

\begin{proof}
By using same arguments as in the proof of Lemma \ref{cv Q(X)}, (\ref%
{representationSigma}) and (\ref{convergence Sigma_T}), we obtain  (\ref%
{bound Q(Sigma)-intSigma}) and (\ref{convergence Q(Sigma)}).\newline
Now, we prove the convergence (\ref{convergence Q(SigmaHat)}). We can write
\begin{eqnarray*}
Q_{n}({\Sigma })-Q_{n}(\widehat{\Sigma }) &=&\frac{-1}{n}\sum_{i=1}^{n}%
\left( {\Sigma }_{t_{i-1}}-\widehat{\Sigma }_{t_{i-1}}\right) ^{2}+\frac{2}{n%
}\sum_{i=1}^{n}{\Sigma }_{t_{i-1}}\left( {\Sigma }_{t_{i-1}}-\widehat{\Sigma
}_{t_{i-1}}\right).
\end{eqnarray*}%
Using the point 6) of Lemma \ref{1mainInequalities}
\begin{eqnarray*}
E\left( \left( {\Sigma }_{t_{i-1}}-\widehat{\Sigma }_{t_{i-1}}\right)
^{2}\right)
&=&\sum_{j=1}^{i-1}\sum_{k=1}^{i-1}\int_{t_{j-1}}^{t_{j}}%
\int_{t_{k-1}}^{t_{k}}E[(X_{s}-X_{t_{j-1}})(X_{r}-X_{t_{k-1}})]drds \\
&\leqslant&c(H,\theta,\rho)\left(\sum_{j=1}^{i-1}\int_{t_{j-1}}^{t_{j}}%
\left|s-t_{j-1}\right|^Hds\right)^2 \\
&\leqslant&c(H,\theta,\rho)\left(n\Delta_n^{H+1}\right)^2
\end{eqnarray*}%
Then, by Hölder inequality and the point 3) of Lemma \ref{1mainInequalities}
we obtain for every $p\geq1$
\begin{equation*}
\left(E\left[\left\vert Q_{n}({\Sigma })-Q_{n}(\widehat{\Sigma })\right\vert
^{p}\right]\right)^{1/p} \leqslant c(H,\theta,\rho) \left[n^2\Delta
_{n}^{2H+2}+n\Delta _{n}^{H+1}\right].
\end{equation*}%
Thus, by (\ref{hypercontractivity}), Lemma \ref{pathwise convergence} and
that fact that $n^{1+\alpha}\Delta_n^{H+1}\rightarrow0$ for some $\alpha>0$
the convergence (\ref{convergence Q(SigmaHat)}) is obtained. \newline
Furthermore, it is also easy to see that the convergence (\ref{convergence
Q(Sigma)-Q(SigmaHat)}) is satisfied.
\end{proof}

\subsection{Auxiliary estimators $\tilde{\protect\theta}$ and $\tilde{%
\protect\rho}$}

The first step in constructing a discrete-observation-based estimator for
which the asymptotics of $\left( \widehat{\theta }_{T},\widehat{\rho }%
_{T}\right) $ studied in Section 3 can be helpful, is to consider the
following two auxiliary estimators $\widetilde{\theta }_{n}$ and $\widetilde{%
\rho }_{n}$ of $\theta ^{\ast }$ and $\rho ^{\ast }$ respectively, by
leaving the numerators in $\left( \widehat{\theta }_{T},\widehat{\rho }%
_{T}\right) $ alone, and discretizing the denominators:%
\begin{equation*}
\widetilde{\theta }_{n}=-\frac{\frac{1}{T_{n}}\int_{0}^{T_{n}}X_{t}\delta
X_{t}}{Q_{n}(X)}
\end{equation*}%
and
\begin{equation*}
\widetilde{\rho }_{n}(\Sigma )=-\frac{\frac{1}{T_{n}}\int_{0}^{T_{n}}%
\widehat{V}_{t}^{n}\delta \widehat{V}_{t}^{n}}{Q_{n}(X)+(\widetilde{\theta }%
_{n})^{2}Q_{n}(\Sigma )},
\end{equation*}%
where
\begin{equation}
\widehat{V}_{t}^{n}=X_{t}+\widehat{\theta }_{n}\Sigma _{t},\quad 0\leqslant
t\leqslant T_{n}.  \label{expression V^n}
\end{equation}%
and we recall that $Q_{n}(Z)$ is a notation for the Riemann-sum rectangle
approximation $\frac{1}{n}\sum_{i=1}^{n}(Z_{t_{i-1}})^{2}$. We also consider
the version of $\widetilde{\rho }_{n}(\Sigma )$ based only on discrete
observations of $\Sigma $:%
\begin{equation*}
\widetilde{\rho }_{n}(\widehat{\Sigma })=-\frac{\frac{1}{T_{n}}%
\int_{0}^{T_{n}}\widehat{V}_{t}^{n}\delta \widehat{V}_{t}^{n}}{Q_{n}(X)+(%
\widetilde{\theta }_{n})^{2}Q_{n}(\widehat{\Sigma })}.
\end{equation*}

Combining Lemma \ref{convergence numerator theta} and the almost-sure
convergence (\ref{convergence Q(X)}) we deduce the strong consistency of $%
\widetilde{\theta }_{n} $.

\begin{theorem}
\label{strong consistency Theta discrete}Assume $H\in (1/2,1)$. If $\Delta
_{n}\leqslant n^{\alpha }\ $for some $\alpha \in (-\infty ,1/H)$, then
\begin{equation*}
\widetilde{\theta }_{n}\longrightarrow \theta ^{\ast }
\end{equation*}%
almost surely as $n\rightarrow \infty $.
\end{theorem}

By Lemmas \ref{convergence numerator of rho} and \ref{cv Q(Sigma)} it is
easy also to deduce the strong consistency of $\widetilde{\rho }_{n}(\Sigma
) $ and $\widetilde{\rho }_{n}(\widehat{\Sigma })$.

\begin{theorem}
\label{strong consistency rho Discrete}Assume $H\in (1/2,1)$. If $\Delta
_{n}\leqslant n^{\alpha }\ $for some $\alpha \in (-\infty ,1/H)$, then%
\begin{equation*}
\widetilde{\rho }_{n}(\Sigma )\longrightarrow \rho ^{\ast }
\end{equation*}%
almost surely as $n\rightarrow \infty $.\newline
In addition, if $n^{1+\alpha }\Delta _{n}^{H+1}\rightarrow 0$ for some $%
\alpha >0$,
\begin{equation*}
\widetilde{\rho }_{n}(\widehat{\Sigma })\longrightarrow \rho ^{\ast }
\end{equation*}%
almost surely as $n\rightarrow \infty $.
\end{theorem}

To establish the asymptotic normality of $\left( \widetilde{\theta }_{n},%
\widetilde{\rho }_{n}\right) $, we can write
\begin{equation*}
\sqrt{T_{n}}\left( \widetilde{\theta }_{n}-\theta ^{\ast }\right) =\frac{%
\frac{S_{T_{n}}}{T_{n}}}{Q_{n}}\sqrt{T_{n}}(\widehat{\theta }_{T_{n}}-\theta
^{\ast })+\frac{\theta ^{\ast }\sqrt{T_{n}}(\frac{S_{T_{n}}}{T_{n}}-{Q_{n}(X)%
})}{Q_{n}(X)}.
\end{equation*}%
\newline
Similarly,
\begin{equation*}
\sqrt{T_{n}}\left( \widetilde{\rho }_{n}-\rho ^{\ast }\right) =\frac{\frac{%
\widehat{L}_{T_{n}}}{T_{n}}}{\widehat{Q}_{n}}\sqrt{T_{n}}(\widehat{\rho }%
_{T_{n}}-\rho ^{\ast })+\frac{\rho ^{\ast }\sqrt{T_{n}}(\frac{\widehat{L}%
_{T_{n}}}{T_{n}}-{\widehat{Q}_{n}})}{\widehat{Q}_{n}}.
\end{equation*}

Theorem \ref{convergence in law theorem} provides the convergence of the
last summands in each of the two lines above. Combining this with the
convergences we obtained in Lemmas \ref{cv Q(X)} and \ref{cv Q(Sigma)}, we
obtain the following result.

\begin{theorem}
\label{convergence in law theorem discrete}Let $H\in (\frac{1}{2},\frac{3}{4}%
)$ and $n\Delta_n^{H+1}\rightarrow0$. Then
\begin{equation*}
\sqrt{T_{n}}\left( \widetilde{\theta }_{n}-\theta ^{\ast },\widetilde{\rho }%
_{n}({\Sigma })-\rho ^{\ast }\right) \overset{law}{\longrightarrow }\mathcal{%
N}\left(0, ^tP\ \Gamma\ P\right)
\end{equation*}%
In addition if $n^{3}\Delta_n^{2H+3}\rightarrow0$,
\begin{equation*}
\sqrt{T_{n}}\left( \widetilde{\theta }_{n}-\theta ^{\ast },\widetilde{\rho }%
_{n}(\widehat{\Sigma })-\rho ^{\ast }\right) \overset{law}{\longrightarrow }%
\mathcal{N}\left(0, ^tP\ \Gamma\ P\right)
\end{equation*}%
where $P $ the matrix defined in (\ref{matrix Sigma}).
\end{theorem}

\subsection{$X$ and $\Sigma $ are observed\label{XandSigma}}

The problem with the auxiliary estimators $\widetilde{\theta }_{n}$ and $%
\widetilde{\rho }_{n}$ is that they still contain Skorohod integrals. In
order to devise a further scheme that allows us to evaluate them, at least
approximately, using discrete data only, we begin by using the discrete
observations of $X$ and $\Sigma $, and recalling that, from Lemmas \ref%
{convergence numerator theta} and \ref{convergence numerator of rho}, we have%
\begin{eqnarray*}
\frac{1}{T}\int_{0}^{T}X_{t}\delta X_{t} &\longrightarrow &-\rho \theta
(\rho +\theta )\eta ^{X}, \\
\frac{1}{T}\int_{0}^{T}\widehat{V}_{t}\delta \widehat{V}_{t}
&\longrightarrow &-\rho \theta (\rho +\theta )\eta ^{\Sigma }
\end{eqnarray*}%
where $\eta ^{X}$ and $\eta ^{\Sigma }$, which are also functions of $%
H,\theta ,\rho $, are given in Lemma \ref{lemme cv ps}. Since these limits
depend on the parameters we are trying to estimate, one strategy is to
rewrite the strong consistency results of Theorems \ref{strong consistency
Theta discrete} and \ref{strong consistency rho Discrete} for $\widetilde{%
\theta }_{n}$ and $\widetilde{\rho }_{n}(\Sigma )$ as implicit definitions
of new estimators, where the numerators in the definitions of $\widetilde{%
\theta }_{n}$ and $\widetilde{\rho }_{n}(\Sigma )$ are replaced by their
limits recalled above, and each instance of $\theta $ and $\rho $ therein
are replaced by the new estimator we are trying to define. The same
substitution must be done with the expressions $\theta ^{\ast }$ and $\rho
^{\ast }$, since these are the limits of $\widetilde{\theta }_{n}$ and $%
\widetilde{\rho }_{n}$. In other words we consider only that the
denominators in $\widetilde{\theta }_{n}$ and $\widetilde{\rho }_{n}$
contain data, and replace all other instances of $\left( \theta ,\rho
\right) $ in the limits in Theorems \ref{strong consistency Theta discrete}
and \ref{strong consistency rho Discrete} by the pair of estimators we are
trying to define. After some minor manipulations, this leads to the
following definition of a new pair of estimators $\left( \check{\theta}_{n},%
\check{\rho}_{n}\right) $ as solution of the system of the following two
equations, if it exists:%
\begin{equation}
\left\{
\begin{array}{lcl}
\check{\theta}_{n}+\check{\rho}_{n}=\frac{H\Gamma (2H)[(\check{\rho}%
_{n})^{2-2H}-(\check{\theta}_{n})^{2-2H}]}{(\check{\rho}_{n}-\check{\theta}%
_{n})Q_{n}(X)} &  &  \\
\frac{\left( \check{\theta}_{n}\right) ^{2}-\left( \check{\rho}_{n}\right)
^{2}}{\left[ (\check{\theta}_{n})^{2-2H}-(\check{\rho}_{n})^{2-2H}+\left(
\check{\rho}_{n}+\check{\theta}_{n}\right) ^{2}\left( (\check{\rho}%
_{n})^{-2H}-(\check{\theta}_{n})^{-2H}\right) \right] }=\frac{H\Gamma (2H)}{%
Q_{n}(X)+(\check{\theta}_{n}+\check{\rho}_{n})^{2}Q_{n}(\Sigma )} &  &
\end{array}%
\right. .  \label{system with Sigma}
\end{equation}%
We emphasize that the above is an implicit definition of $\left( \check{%
\theta}_{n},\check{\rho}_{n}\right) $. It is also rather opaque. The system (%
\ref{system with Sigma}) can be simplified slightly using more elementary
manipulations. We find that the definition of $\left( \check{\theta}_{n},%
\check{\rho}_{n}\right) $ is equivalent to the following:%
\begin{equation*}
F\left( \check{\theta}_{n},\check{\rho}_{n}\right) =\left(
Q_{n}(X),Q_{n}(\Sigma )\right)
\end{equation*}%
where $F$ is a positive function of the variables $\left( x,y\right) $ in $%
(0,+\infty )^{2}$ defined by: for every $(x,y)\in (0,+\infty )^{2}$
\begin{equation}
F(x,y)=H\Gamma (2H)\times \left\{
\begin{array}{lcl}
\frac{1}{y^{2}-x^{2}}\left( y^{2-2H}-x^{2-2H},x^{-2H}-y^{-2H}\right) \quad %
\mbox{if }\ x\neq y &  &  \\
\left( (1-H)x^{-2H},Hx^{-2H-2}\right) \quad \mbox{if }\ x=y. &  &
\end{array}%
\right.  \label{new system with Sigma}
\end{equation}%
Interestingly, this shows that a good candidate for the discrete version of
the least-squares estimator of $\left( \theta ,\rho \right) $ is none other
than a type of generalized method of moments estimator obtained via Lemma %
\ref{lemme cv ps} after discretizing the expressions $S_{T}({X}%
):=T^{-1}\int_{0}^{T}X_{s}^{2}ds$ and $S_{T}({\Sigma }):=T^{-1}\int_{0}^{T}%
\Sigma _{s}^{2}ds$. We now consider the question whether System (\ref{system
with Sigma})  has a unique solution $\left( \check{\theta}_{n}^{2},\check{%
\rho}_{n}^{2}\right) $, and how this may imply strong consistency for these
estimators.\newline
Since for every $(x,y)\in (0,+\infty )^{2}$ with $x\neq y$
\begin{equation*}
J_{F}\left( x,y\right) =\Gamma (2H+1)%
\begin{pmatrix}
\frac{\left( 1-H\right) x^{1-2H}\left( x^{2}-y^{2}\right) -x\left(
x^{2-2H}-y^{2-2H}\right) }{\left( x^{2}-y^{2}\right) ^{2}} & \frac{\left(
1-H\right) y^{1-2H}\left( y^{2}-x^{2}\right) -y\left(
y^{2-2H}-x^{2-2H}\right) }{\left( x^{2}-y^{2}\right) ^{2}} \\
\frac{Hx^{-2H-1}\left( x^{2}-y^{2}\right) +x\left( x^{-2H}-y^{-2H}\right) }{%
\left( x^{2}-y^{2}\right) ^{2}} & \frac{Hy^{-2H-1}\left( y^{2}-x^{2}\right)
+y\left( y^{-2H}-x^{-2H}\right) }{\left( x^{2}-y^{2}\right) ^{2}}%
\end{pmatrix}%
\end{equation*}%
the determinant of $J_{F}\left( x,y\right) $ is non-zero on in $(0,+\infty
)^{2}$. So, $F$ is a diffeomorphism in $(0,+\infty )^{2}$ and its inverse $G$
has a Jacobian%
\begin{equation*}
J_{G}\left( a,b\right) =\frac{\Gamma (2H+1)}{\det J_{F}\left( x,y\right) }%
\begin{pmatrix}
\frac{Hy^{-2H-1}\left( y^{2}-x^{2}\right) +y\left( y^{-2H}-x^{-2H}\right) }{%
\left( x^{2}-y^{2}\right) ^{2}} & -\frac{\left( 1-H\right) y^{1-2H}\left(
y^{2}-x^{2}\right) -y\left( y^{2-2H}-x^{2-2H}\right) }{\left(
x^{2}-y^{2}\right) ^{2}} \\
-\frac{Hx^{-2H-1}\left( x^{2}-y^{2}\right) +x\left( x^{-2H}-y^{-2H}\right) }{%
\left( x^{2}-y^{2}\right) ^{2}} & \frac{\left( 1-H\right) x^{1-2H}\left(
x^{2}-y^{2}\right) -x\left( x^{2-2H}-y^{2-2H}\right) }{\left(
x^{2}-y^{2}\right) ^{2}}%
\end{pmatrix}%
;
\end{equation*}%
where $\left( x,y\right) =G\left( a,b\right) $.\newline
Hence, (\ref{convergence Q(X)}) and (\ref{convergence Q(Sigma)}) lead to
\begin{equation*}
\left( \check{\theta}_{n},\check{\rho}_{n}\right) =G\left(
Q_{n}(X),Q_{n}(\Sigma )\right) \longrightarrow G\left( \eta ^{X},\eta
^{\Sigma }\right) =\left( \theta ,\rho \right)
\end{equation*}%
almost surely as $n\rightarrow \infty $ as soon as $\Delta _{n}\leqslant
n^{\alpha }\ $for some $\alpha \in (-\infty ,1/H)$. Summarizing, we have
proved the following.

\begin{theorem}
Let $H\in (1/2,1)$ and assume that $\Delta _{n}\leqslant n^{\alpha }\ $for
some $\alpha \in (-\infty ,1/H)$. Then, as $n\longrightarrow \infty $
\begin{equation*}
\left( \check{\theta}_{n},\check{\rho}_{n}\right) \longrightarrow (\theta
,\rho )
\end{equation*}%
almost surely.
\end{theorem}

We may now prove a normal convergence result for $\left( \check{\theta}_{n},%
\check{\rho}_{n}\right) $ based on Theorem \ref{convergence in law theorem
discrete}. Note that the second part of Theorem \ref{convergence in law
theorem discrete} is not needed here because we rely on fully observed $%
\Sigma $ in this section.

\begin{theorem}
\label{convergence in law of check theta}Suppose that $H\in (\frac{1}{2},%
\frac{3}{4})$ and $n\Delta _{n}^{H+1}\rightarrow 0$. Then
\begin{equation*}
\sqrt{T_{n}}\left( \check{\theta}_{n}-\theta ,\check{\rho}_{n}-\rho \right)
\overset{law}{\longrightarrow }\mathcal{N}(0,^{t}M\ ^{t}P\ \Gamma \ P\ M)
\end{equation*}%
where the matrices $\Gamma $, $P$ and $M$ are defined respectively in (\ref%
{matrix Gamma}), (\ref{matrix Sigma}) and (\ref{matrix M}).
\end{theorem}

\begin{proof}
We have
\begin{equation*}
\widetilde{\theta }_{n}=-\frac{\frac{1}{T_{n}}\int_{0}^{T_{n}}X_{t}\delta
X_{t}}{Q_{n}(X)}:=\frac{J_{n}^{\theta }(X)}{Q_{n}(X)}
\end{equation*}%
and
\begin{equation*}
\widetilde{\rho }_{n}(\Sigma )=-\frac{\frac{1}{T_{n}}\int_{0}^{T_{n}}%
\widehat{V}_{t}^{n}\delta \widehat{V}_{t}^{n}}{Q_{n}(X)+(\widetilde{\theta }%
_{n})^{2}Q_{n}(\Sigma )}:=\frac{J_{n}^{\rho }(\widehat{V})}{Q_{n}(X)+(%
\widetilde{\theta }_{n})^{2}Q_{n}(\Sigma )}.
\end{equation*}%
Then, we can write
\begin{eqnarray*}
&&\sqrt{T_{n}}\left( \check{\theta}_{n}-\theta ,\check{\rho}_{n}-\rho
)\right) \\
&=&\sqrt{T_{n}}\left( G\left( \frac{J_{n}^{\theta }(X)}{\widetilde{\theta }%
_{n}},(\widetilde{\theta }_{n})^{-2}\left( \frac{J_{n}^{\rho }(\widehat{V})}{%
\widetilde{\rho }_{n}(\Sigma )}-\frac{J_{n}^{\theta }(X)}{\widetilde{\theta }%
_{n}}\right) \right) -(\theta ,\rho )\right) \\
&=&\sqrt{T_{n}}\left( GoL\left( J_{n}^{\theta }(X),J_{n}^{\rho }(\widehat{V}%
),\widetilde{\theta }_{n},\widetilde{\rho }_{n}(\Sigma )\right) -(\theta
,\rho )\right)
\end{eqnarray*}%
where $L(r,s,u,v)=\left( \frac{r}{u},\frac{s}{u^{2}v}-\frac{r}{u^{3}}\right)
$.\newline
On the other hand for any $\varepsilon \in (0,1)$
\begin{equation*}
\frac{1}{T}\int_{0}^{T}\int_{0}^{t}r^{2H-2}e^{-r}drdt=\Gamma (2H-1)+o(\frac{1%
}{T^{\varepsilon }})
\end{equation*}%
because
\begin{eqnarray*}
\frac{1}{T^{1-\varepsilon }}\int_{0}^{T}\int_{t}^{\infty }r^{2H-2}e^{-r}drdt
&\leqslant &\frac{1}{2T^{1-\varepsilon }}(1-e^{-T/2})\int_{0}^{\infty
}r^{2H-2}e^{-r/2}dr \\
&\rightarrow &0.
\end{eqnarray*}%
Combining this together with (\ref{decomposition of int_X_delta_X}) and the
point 5) of Lemma \ref{1mainInequalities} we can write
\begin{equation}
J_{n}^{\theta }(X)=(\rho +\theta )\eta ^{X}+o(\frac{1}{\sqrt{T_{n}}})
\label{estimation1}
\end{equation}%
where $o(\frac{1}{\sqrt{T_{n}}})$ denotes a random variable such that $\sqrt{%
T_{n}}o(\frac{1}{\sqrt{T_{n}}})$ converges to zero almost surely as $%
T_{n}\rightarrow \infty $.\newline
Similar argument leads to
\begin{equation}
J_{n}^{\rho }(\widehat{V})=\rho \theta (\rho +\theta )\eta ^{\Sigma }+o(%
\frac{1}{\sqrt{T_{n}}}).  \label{estimation2}
\end{equation}%
Since
\begin{equation*}
GoL\left( (\rho +\theta )\eta ^{X},\rho \theta (\rho +\theta )\eta ^{\Sigma
},\theta ^{\ast },\rho ^{\ast }\right) =(\theta ,\rho )
\end{equation*}%
we can write
\begin{eqnarray*}
&&\sqrt{T_{n}}\left( \check{\theta}_{n}-\theta ,\check{\rho}_{n}-\rho
)\right) \\
&=&\sqrt{T_{n}}\left( GoL\left( J_{n}^{\theta }(X),J_{n}^{\rho }(\widehat{V}%
),\widetilde{\theta }_{n},\widetilde{\rho }_{n}(\Sigma )\right) -GoL\left(
(\rho +\theta )\eta ^{X},\rho \theta (\rho +\theta )\eta ^{\Sigma },\theta
^{\ast },\rho ^{\ast }\right) \right) \\
&=&\sqrt{T_{n}}\left[ GoL\left( J_{n}^{\theta }(X),J_{n}^{\rho }(\widehat{V}%
),\widetilde{\theta }_{n},\widetilde{\rho }_{n}(\Sigma )\right) -GoL\left(
(\rho +\theta )\eta ^{X},\rho \theta (\rho +\theta )\eta ^{\Sigma },%
\widetilde{\theta }_{n},\widetilde{\rho }_{n}(\Sigma )\right) \right. \\
&&+\left. GoL\left( (\rho +\theta )\eta ^{X},-\rho \theta (\rho +\theta
)\eta ^{\Sigma },\widetilde{\theta }_{n},\widetilde{\rho }_{n}(\Sigma
)\right) -GoL\left( (\rho +\theta )\eta ^{X},\rho \theta (\rho +\theta )\eta
^{\Sigma },\theta ^{\ast },\rho ^{\ast }\right) \right] \\
:= &&s_{n}+r_{n}.
\end{eqnarray*}%
From (\ref{estimation1}) and (\ref{estimation2}) we obtain $%
s_{n}\longrightarrow 0$ almost surely as $n\rightarrow \infty $.\newline
On the other hand, by Taylor's formula
\begin{equation*}
r_{n}=\sqrt{T_{n}}\left( \widetilde{\theta }_{n}-\theta ^{\ast },\widetilde{%
\rho }_{n}(\Sigma )-\rho ^{\ast }\right) M+d_{n}
\end{equation*}%
where%
\begin{equation}
M=\left(
\begin{matrix}
\frac{\partial h_{1}}{\partial u}(\theta ^{\ast },\rho ^{\ast }) & \frac{%
\partial h_{2}}{\partial u}(\theta ^{\ast },\rho ^{\ast }) \\
\frac{\partial h_{1}}{\partial v}(\theta ^{\ast },\rho ^{\ast }) & \frac{%
\partial h_{2}}{\partial v}(\theta ^{\ast },\rho ^{\ast })%
\end{matrix}%
\right)  \label{matrix M}
\end{equation}%
with
\begin{eqnarray*}
h(u,v)=(h_{1},h_{2})(u,v) &=&GoL\left( (\rho +\theta )\eta ^{X},\rho \theta
(\rho +\theta )\eta ^{\Sigma },u,v\right).
\end{eqnarray*}%
We can write
\begin{eqnarray*}
h(u,v)&= &\left( {G_{1}},{G_{2}}\right) og\left( u,v\right)
\end{eqnarray*}%
where
\begin{equation*}
g(u,v)=(g_{1},g_{2})(u,v)=L\left( (\rho +\theta )\eta ^{X},\rho \theta (\rho
+\theta )\eta ^{\Sigma },u,v\right) .
\end{equation*}%
Moreover for $i=1,2$
\begin{equation*}
\frac{\partial h_{i}}{\partial u}(u,v)=\frac{\partial G_{i}}{\partial a}%
(g(u,v))\frac{\partial g_{1}}{\partial u}(u,v)+\frac{\partial G_{i}}{%
\partial b}(g(u,v))\frac{\partial g_{2}}{\partial u}(u,v)
\end{equation*}%
and
\begin{equation*}
\frac{\partial h_{i}}{\partial v}(u,v)=\frac{\partial G_{i}}{\partial a}%
(g(u,v))\frac{\partial g_{1}}{\partial v}(u,v)+\frac{\partial G_{i}}{%
\partial b}(g(u,v))\frac{\partial g_{2}}{\partial v}(u,v).
\end{equation*}%
On the other hand, $d_{n}$ converges in distribution to zero, because
\begin{equation*}
\Vert d_{n}\Vert \leqslant c(H, \theta ,\rho)\sqrt{T_{n}}\Vert (\widetilde{%
\rho }_{n}(\Sigma )-\theta ^{\ast },\widetilde{\theta }_{n}-\rho ^{\ast
})\Vert ^{2}.
\end{equation*}%
It is elementary that if for any $\omega \in \Omega $ there exists $%
n_{0}(\omega )\in \mathbb{N}$ such that $X_{n}(\omega )=Y_{n}(\omega )$ for
all $n\geq n_{0}(\omega )$ and $X_{n}\overset{law}{\longrightarrow }0$ as $%
n\rightarrow \infty $, then $Y_{n}\overset{law}{\longrightarrow }0$ as $%
n\rightarrow \infty $.\newline
Combining this with Theorem \ref{convergence in law theorem discrete} the
proof is completed.
\end{proof}

\subsection{$X$ is observed\label{Xobserved}}

In the previous section, we encountered theorems in which $X$ and $\Sigma $
are both assumed to be fully observed in discrete time. Since $\Sigma $ is
the time-antiderivative of $X$, such an assumption corresponds, for
instance, to the physical situation where $X$ is the velocity of a particle,
and $\Sigma $ is its position.

In this section, we abandon such a framework, and assume instead that only $%
X $ is observed in discrete time. Thus we consider the following pair of
estimators $\left( \breve{\theta}_{n},\breve{\rho}_{n}\right) $:%
\begin{equation*}
\left( \breve{\theta}_{n},\breve{\rho}_{n}\right) =G\left( Q_{n}(X),Q_{n}(%
\widehat{\Sigma })\right)
\end{equation*}%
where the deterministic explicit function $G$ was identified in the previous
section as the inverse of the function $F$ given in (\ref{new system with
Sigma}). Equivalently, $\left( \breve{\theta}_{n},\breve{\rho}_{n}\right) $
is the solution of the system (\ref{system with Sigma}), or its equivalent
form (\ref{new system with Sigma}), with $\Sigma $ replaced by the process $%
\widehat{\Sigma } $, which relies only on observations of $X$. Using same
arguments as in Section \ref{XandSigma} and Lemma \ref{cv Q(Sigma)}, but
relying now on the second part of Theorem \ref{strong consistency rho
Discrete} (hence the stronger condition on $\Delta _{n}$ for the strong
consistency result) and the second part of Theorem \ref{convergence in law
theorem discrete} (hence the stronger condition on $\Delta _{n}$ for the
convergence in law result), we conclude the following.

\begin{theorem}
\label{convergence in law of check theta case hat Sigma} If $n^{1+\alpha
}\Delta _{n}^{H+1}\rightarrow 0$ for some $\alpha >0$,
\begin{equation*}
\left( \breve{\theta}_{n},\breve{\rho}_{n}\right) \longrightarrow \left(
\theta ,\rho \right)
\end{equation*}%
almost surely as $n\rightarrow \infty $.
\end{theorem}

\begin{theorem}
\label{convergence in law of check theta case hat Sigma}Let $H\in (\frac{1}{2%
},\frac{3}{4})$. If $n^{3}\Delta _{n}^{2H+3}\rightarrow 0$, then, as $%
n\rightarrow \infty $
\begin{equation*}
\sqrt{T_{n}}\left( \breve{\theta}_{n}-\theta ,\breve{\rho}_{n}-\rho \right)
\overset{law}{\longrightarrow }\mathcal{N}(0,^{t}M\ ^{t}P\ \Gamma \ P\ M)
\end{equation*}%
where the matrices $\Gamma $, $P$ and $M$ are defined respectively in (\ref%
{matrix Gamma}), (\ref{matrix Sigma}) and (\ref{matrix M}).
\end{theorem}

\subsection{$X$ and $V$ are observed}

When both $X$ and $V$ are observed, the estimator of $\theta ^{\ast }$ based
on continuous data is $\widehat{\theta }_{T}$ given in (\ref%
{theta^fractional case}) but the estimator of $\rho $ becomes the usual
full-observation estimator of an fBm-driven Ornstein-Uhlenbeck process as in
\cite{HN}, i.e.%
\begin{equation*}
\overline{\rho }_{T}=-\frac{\int_{0}^{T}{V}_{t}\delta {V}_{t}}{\int_{0}^{T}{V%
}_{t}^{2}dt}.
\end{equation*}%
Following similar arguments as in the beginning of Section \ref{XandSigma},
the natural candidate for the estimator based on discrete data of $X$ and $V$
is the pair $\left( \underline{\theta }_{n},\underline{\rho }_{n}\right) $
defined as the solution of the following system:%
\begin{equation*}
\left\{
\begin{array}{lcl}
\underline{\theta }_{n}+\underline{\rho }_{n}=\frac{\frac{H\Gamma (2H)}{(%
\underline{\rho }_{n}-\underline{\theta }_{n})}[(\underline{\rho }%
_{n})^{2-2H}-(\underline{\theta }_{n})^{2-2H}]}{Q_{n}(X)} &  &  \\
\underline{\rho }_{n}=\left( \frac{H\Gamma (2H)}{Q_{n}(V)}\right) ^{\frac{1}{%
2H}}. &  &
\end{array}%
\right.  \label{new system with X and V observed}
\end{equation*}%
We see that $\underline{\rho }_{n}$ is defined explicitly autonomously via
the discrete-data-based statistic $Q_{n}(V)$. With $\underline{\rho }_{n}$
now known, elementary manipulations yield that $\underline{\theta }_{n}$ is
precisely the solution of the following simple equation%
\begin{equation*}
\left( \underline{\theta }_{n}\right) ^{2-2H}-\left( \frac{Q_{n}(X)}{H\Gamma
(2H)}\right)\left( \underline{\theta }_{n}\right) ^{2}=\left( \underline{%
\rho }_{n}\right) ^{2-2H}-\left( \frac{Q_{n}(X)}{H\Gamma (2H)}\right)\left(
\underline{\rho }_{n}\right) ^{2}.
\end{equation*}%
Define
\begin{equation*}
\overline{F}(x,y)=H\Gamma(2H)\times\left\{
\begin{array}{lcl}
\left(\frac{y^{2-2H}-x^{2-2H}}{y^2-x^2}, y^{-2H}\right)\quad\mbox{if }\
x\neq y &  &  \\
\left((1-H)x^{-2H}, x^{-2H}\right)\quad\mbox{if }\ x=y. &  &
\end{array}%
\right.
\end{equation*}
Its Jacobian is given, for every $(x,y)\in(0,+\infty )^{2}$ such that $x\neq
y$, by%
\begin{equation*}
J_{\overline{F}}\left( x,y\right) =\Gamma(2H+1)%
\begin{pmatrix}
\frac{\left( 1-H\right) x^{1-2H}\left( x^2-y^2\right)
-x\left(x^{2-2H}-y^{2-2H}\right)}{\left( x^2-y^2\right) ^{2}} & \frac{\left(
1-H\right) y^{1-2H}\left( y^2-x^2\right) -y\left(y^{2-2H}-x^{2-2H}\right)}{%
\left( x^2-y^2\right) ^{2}} \\
0 & -y^{-2H-1}%
\end{pmatrix}%
.  \label{Jacobian overline{F}}
\end{equation*}
Thus the Jacobian of is inverse $\overline{G}$ is as follows
\begin{equation*}
J_{\overline{G}}\left( a,b\right) =\frac{\Gamma(2H+1)}{\det J_{\overline{F}%
}\left(x,y\right) }%
\begin{pmatrix}
-y^{-2H-1} & -\frac{\left( 1-H\right) y^{1-2H}\left( y^2-x^2\right)
-y\left(y^{2-2H}-x^{2-2H}\right)}{\left( x^2-y^2\right) ^{2}} \\
0 & \frac{\left( 1-H\right) x^{1-2H}\left( x^2-y^2\right)
-x\left(x^{2-2H}-y^{2-2H}\right)}{\left( x^2-y^2\right) ^{2}}%
\end{pmatrix}%
;\left( x,y\right) =\overline{G}\left( a,b\right).
\end{equation*}%
Using same arguments as in Section \ref{XandSigma} we obtain

\begin{theorem}
\label{convergence in law of check theta case X,V}Assume that $\Delta
_{n}\leqslant n^{\alpha }\ $for some $\alpha \in (-\infty ,1/H)$. Then, as $%
n\longrightarrow \infty $
\begin{equation*}
\left( \overline{\theta }_{n},\overline{\rho }_{n}\right) \longrightarrow
\left( \theta ,\rho \right)
\end{equation*}%
almost surely.
\end{theorem}

\begin{theorem}
\label{convergence in law of check theta case X,V}Suppose that $%
H\in(\frac12,\frac34)$ and $n\Delta_n^{H+1}\rightarrow0$. Then, as $%
n\rightarrow \infty $
\begin{equation*}
\sqrt{T_{n}}\left( \overline{\theta }_{n}-\theta ,\overline{\rho }_{n}-\rho
\right) \overset{law}{\longrightarrow }\mathcal{N}(0,^tQ\ ^tP\ \Gamma\ P\ Q )
\end{equation*}

where $\Gamma $ and $P$ are defined respectively in (\ref{matrix Gamma}) and
(\ref{matrix Sigma}), and where
\begin{equation*}
Q=\left(
\begin{matrix}
\frac{\partial f_{1}}{\partial u}(\theta ^{\ast },\rho ) & \frac{\partial
f_{2}}{\partial u}(\theta ^{\ast },\rho ) \\
\frac{\partial f_{1}}{\partial v}(\theta ^{\ast },\rho ) & \frac{\partial
f_{2}}{\partial v}(\theta ^{\ast },\rho )%
\end{matrix}%
\right)
\end{equation*}%
such that
\begin{equation*}
(f_{1},f_{2})(u,v)=\overline{G}ol\left( u,v\right)
\end{equation*}%
with
\begin{equation*}
l(u,v)=\overline{L}\left( (\rho +\theta )\eta ^{X},\rho ^{1-2H},u,v\right)
\end{equation*}%
and
\begin{equation*}
\overline{L}(r,s,u,v):=\left( \frac{r}{u},\frac{s}{v}\right) .
\end{equation*}
\end{theorem}

\renewcommand{\theequation}{A-\arabic{equation}}
\setcounter{equation}{0} 

\section{Appendix}

In this appendix, we present some calculations used in the paper.\vspace*{%
0.12in}

Fix $T>0$. Let $f,g:[0,T]\longrightarrow \mathbb{R}$ {\ be} Hölder
continuous functions of orders $\alpha \in (0,1)$ and $\beta \in (0,1)$
respectively with $\alpha +\beta >1$. Young \cite{Young} proved that the
Riemann-Stieltjes integral (so-called Young integral) $%
\int_{0}^{T}f_{s}dg_{s}$ exists. Moreover, if $\alpha =\beta \in (\frac{1}{2}%
,1)$ and $\phi :\mathbb{R}^{2}\longrightarrow \mathbb{R}$ is a function of
class $\mathcal{C}^{1}$, the integrals $\int_{0}^{.}\frac{\partial \phi }{%
\partial f}(f_{u},g_{u})df_{u}$ and $\int_{0}^{.}\frac{\partial \phi }{%
\partial g}(f_{u},g_{u})dg_{u}$ exist in the Young sense and the following
chain rule holds:
\begin{equation}
\phi (f_{t},g_{t})=\phi (f_{0},g_{0})+\int_{0}^{t}\frac{\partial \phi }{%
\partial f}(f_{u},g_{u})df_{u}+\int_{0}^{t}\frac{\partial \phi }{\partial g}%
(f_{u},g_{u})dg_{u},\quad 0\leqslant t\leqslant T.  \label{chain
rule}
\end{equation}%
As a consequence, if $H\in (\frac{1}{2},1)$ and $(u_{t},\ t\in \lbrack 0,T])$
{\ is } a process with Hölder paths of order $\alpha \in (1-H,1)$, the
integral $\int_{0}^{T}u_{s}dB_{s}^{H}$ is well-defined as {\ a} Young
integral. Suppose moreover that for any $t\in \lbrack 0,T]$, $u_{t}\in
D^{1,2}$, and
\begin{equation*}
P\left( \int_{0}^{T}\int_{0}^{T}|D_{s}u_{t}||t-s|^{2H-2}dsdt<\infty \right)
=1.
\end{equation*}%
Then, by \cite{AN}, $u\in Dom\delta $ and for every $t\in \lbrack 0,T]$,
\begin{equation}  \label{link}
\int_{0}^{t}u_{s}dB_{s}^{H}=\int_{0}^{t}u_{s}\delta
B_{s}^{H}+H(2H-1)\int_{0}^{t}\int_{0}^{t}D_{s}u_{r}|s-r|^{2H-2}drds.
\end{equation}%
In particular, when $\varphi $ is a non-random Hölder continuous function of
order $\alpha \in (1-H,1)$, we obtain
\begin{equation}
\int_{0}^{T}\varphi _{s}dB_{s}^{H}=\int_{0}^{T}\varphi _{s}\delta
B_{s}^{H}=B^{H}(\varphi ).  \label{non random}
\end{equation}%
\newline
In addition, for all $\varphi ,\ \psi \in |\mathcal{H}|$,
\begin{equation*}
E\left( \int_{0}^{T}\varphi _{s}dB_{s}^{H}\int_{0}^{T}\psi
_{s}dB_{s}^{H}\right) =H(2H-1)\int_{0}^{T}\int_{0}^{T}\varphi (u)\psi
(v)|u-v|^{2H-2}dudv.
\end{equation*}

\begin{lemma}
\label{1mainInequalities}Let $m, m^{\prime}>0$ and let $X^m$ be the process
defined in (\ref{Xm}). Then,

\begin{enumerate}

\item [1)] $\lambda (m,m^{\prime }):=H(2H-1)\int_{0}^{\infty }\int_{0}^{\infty
}e^{-ms}e^{-m^{\prime }r}|s-r|^{2H-2}drds=\frac{H\Gamma (2H)}{m+m^{\prime }}%
\left( m^{1-2H}+{m^{\prime }}^{1-2H}\right) $,

\item [2)] $0\leqslant \lambda (m,m^{\prime })-E(X_{t}^{m}X_{t}^{m^{\prime
}})\leqslant c(H,m,m^{\prime })e^{-t/2}$,

\item [3)] $\sup_{t\geq 0}E[|X_{t}^{m}|^{p}]\leqslant c(H,m,p)<\infty $,

\item [4)] $0\leqslant E(X_{t}^{m}X_{s}^{m^{\prime }})\leqslant c(H,m,m^{\prime
})|t-s|^{2H-2},$

\item [5)] For every $\varepsilon >0,\frac{X_{T}^{m}}{T^{\varepsilon }}%
\rightarrow 0$ almost surely as $T\rightarrow \infty $,

\item [6)] $E(|X_{t}^{m}-X_{s}^{m}|^{p})\leqslant c(H,m,p)|t-s|^{pH}$.
\end{enumerate}
\end{lemma}

\begin{proof}
To prove equality 1), we just write \newline
\begin{eqnarray*}
\lambda (m,m^{\prime }) &=&H(2H-1)\int_{0}^{\infty }\int_{0}^{\infty
}e^{-ms}e^{-m^{\prime }r}|s-r|^{2H-2}drds \\
&=&H(2H-1)\int_{0}^{\infty }dse^{-ms}\int_{0}^{s}dre^{-m^{\prime
}r}|s-r|^{2H-2} \\
&&+H(2H-1)\int_{0}^{\infty }dse^{-ms}\int_{s}^{\infty }dre^{-m^{\prime
}r}|s-r|^{2H-2} \\
&=&\frac{H\Gamma (2H)}{m+m^{\prime }}\left( m^{1-2H}+{m^{\prime }}%
^{1-2H}\right) .
\end{eqnarray*}
For the point 2) see \cite{HS}. For 3) and 6) we refer to \cite{GKN}, and
for 4) and 5) see \cite[Lemma 5.2 and Lemma 5.4]{HN}
\end{proof}

\begin{theorem}
\label{key asymptotic normality} Let $H\in(\frac12,\frac34)$. Define for $%
m>0 $%
\begin{equation*}
f_T^{m}(u,v) :=\frac{1}{2}e^{-m|u-v|}\leavevmode%
\hbox{\rm
\small1\kern-0.35em\normalsize1}_{\{[0,T]\}}^{\otimes2}(u,v).
\end{equation*}
Then, as $T\longrightarrow \infty$,
\begin{eqnarray}
&&\frac{1}{\sqrt{T}}\left( I_2(f^{\theta}_T), I_2(f^{\rho}_T) \right)\overset%
{law}{\longrightarrow} \mathcal{N}(0,\Gamma)  \label{convergence in
law key}
\end{eqnarray}
where $\Gamma$ is a symmetric nonnegative definite matrix which has the
following explicit expression
\begin{eqnarray}  \label{matrix Gamma}
\Gamma=\eta_H\left(%
\begin{matrix}
l_1 & l_3 \\
l_3 & l_2%
\end{matrix}%
\right)
\end{eqnarray}%
where $l_1=\theta^{1-4H}$, $l_2= \rho^{1-4H} $, $l_3=\frac{2\rho\theta}{%
(4H-1)(\rho^2-\theta^2)}\left[\theta^{1-4H}-\rho^{1-4H}\right]$ and
\begin{equation*}
\eta_H=H^2(4H-1)\left[\Gamma(2H)^2+\frac{\Gamma(2H)\Gamma(3-4H)\Gamma(4H-1)}{%
\Gamma(2-2H)}\right].
\end{equation*}
\end{theorem}

\begin{proof}
Notice that for (\ref{convergence in law key}) to hold it suffices that
prove that for every $a,b\in \mathbb{R}$,
\begin{equation*}
G_{T}:=aI_{2}(f_{T}^{\theta })+bI_{2}(f_{T}^{\rho })
\end{equation*}%
converges in law to $\mathcal{N}\left( 0,(a,b)\Gamma ^{t}(a,b)\right) $ as $%
T\longrightarrow \infty $.\newline
Fix $a,b\in \mathbb{R}$. Since $G_{T}$ is a multiple integral, by the
isometry property of double stochastic integral $I_{2}$, we get the variance
of $G_{T}$ as follows
\begin{equation*}
EG_{T}^{2}=\frac{\alpha _{H}^{2}}{2}\left( a^{2}\frac{I_{T}^{1}}{T}+b^{2}%
\frac{I_{T}^{2}}{T}+2ab\frac{I_{T}^{3}}{T}\right) ,
\end{equation*}%
where
\begin{equation*}
I_{T}^{1}=\int_{[0,T]^{4}}e^{-\theta |t-s|}e^{-\theta
|u-v|}|t-u|^{2H-2}|s-v|^{2H-2}dtdsdudv,
\end{equation*}%
\begin{equation*}
I_{T}^{2}=\int_{[0,T]^{4}}e^{-\rho |t-s|}e^{-\rho
|u-v|}|t-u|^{2H-2}|s-v|^{2H-2}dtdsdudv,
\end{equation*}%
\begin{equation*}
I_{T}^{3}=\int_{[0,T]^{4}}e^{-\rho |t-s|}e^{-\theta
|u-v|}|t-u|^{2H-2}|s-v|^{2H-2}dtdsdudv.
\end{equation*}%
Using the same argument as in the proof of \cite[Theorem 3.4]{HN}, we have
as $T\rightarrow \infty $
\begin{equation*}
lim_{T\rightarrow \infty }\frac{\alpha _{H}^{2}}{2}\frac{I_{T}^{1}}{T}=\eta
_{H}l_{1}
\end{equation*}%
and%
\begin{equation*}
lim_{T\rightarrow \infty }\frac{\alpha _{H}^{2}}{2}\frac{I_{T}^{2}}{T}=\eta
_{H}l_{2}.
\end{equation*}%
Now, let us estimate $I_{T}^{3}$. We have
\begin{eqnarray*}
\frac{dI_{T}^{3}}{dT} &=&2\left[ \int_{[0,T]^{3}}e^{-\rho (T-s)}e^{-\theta
|u-v|}(T-u)^{2H-2}|s-v|^{2H-2}dsdudv\right] \\
&&+2\left[ \int_{[0,T]^{3}}e^{-\rho |t-s|}e^{-\theta
(T-v)}(T-t)^{2H-2}|s-v|^{2H-2}dsdudv\right] \\
:= &&A_{T}(\rho ,\theta )+A_{T}(\theta ,\rho ).
\end{eqnarray*}%
Making the change of variables $T-s=x$ , $T-u=y$ and $T-v=z$
\begin{equation*}
A_{T}(\rho ,\theta )=2\int_{[0,T]^{3}}e^{-\rho x}e^{-\theta
|y-z|}y^{2H-2}|z-x|^{2H-2}dxdydz.
\end{equation*}%
This implies
\begin{equation*}
lim_{T\rightarrow \infty }A_{T}(\rho ,\theta )=A_{\infty }(\rho ,\theta
)=2\int_{[0,\infty ]^{3}}e^{-\rho x}e^{-\theta
|y-z|}y^{2H-2}|z-x|^{2H-2}dxdydz.
\end{equation*}%
Making the change of variables $z-x=w$, we obtain
\begin{eqnarray*}
A_{\infty }(\rho ,\theta ) &=&2\left[ \int_{[0,\infty
)^{2}}\int_{-x}^{\infty }e^{-\rho x}e^{-\theta
|y-w-x|}y^{2H-2}|w|^{2H-2}dwdxdy\right] \\
&=&2\left[ \int_{[0,\infty )^{2}}\int_{-x}^{y-x}e^{-\rho x}e^{-\theta
(y-w-x)}y^{2H-2}|w|^{2H-2}dwdxdy\right] \\
&&+2\left[ \int_{(0,\infty )^{2}}\int_{y-x}^{\infty }e^{-\rho x}e^{\theta
(y-w-x)}y^{2H-2}|w|^{2H-2}dwdxdy\right] .
\end{eqnarray*}%
Integrating in $x$ we get
\begin{eqnarray*}
A_{\infty }(\rho ,\theta ) &=&\frac{2}{(-\rho +\theta )}\left[
\int_{0}^{\infty }\int_{-\infty }^{+\infty }(e^{(-\rho +\theta
)(y-w)}-e^{(-\rho +\theta )[(-w)\vee 0]})1_{[(y-w)-((-w)\vee 0)]_{+}}\right.
\\
&&\left. \times e^{-\theta (y-w)}y^{2H-2}|w|^{2H-2}\phantom{\int}dwdy\right]
\\
&&-\frac{2}{(-\rho -\theta )}\left[ \int_{0}^{\infty }\int_{-\infty
}^{+\infty }e^{(-\theta -\rho )[(y-w)\vee 0]}e^{\theta
(y-w)}y^{2H-2}|w|^{2H-2}dwdy\right] \\
&=&\frac{2}{\theta -\rho }A^{1}(\rho ,\theta )+\frac{2}{\rho +\theta }%
A^{2}(\rho ,\theta ).
\end{eqnarray*}%
Furthermore
\begin{eqnarray*}
A^{1}(\rho ,\theta ) &=&\int_{0}^{\infty }\int_{0}^{\infty }(e^{-\rho
(y+w)}-e^{-\rho w}e^{-\theta y})y^{2H-2}w^{2H-2}dwdy \\
&&+\int_{0}^{\infty }\int_{0}^{y}(e^{-\rho (y-w)}-e^{-\theta
(y-w)})y^{2H-2}w^{2H-2}dwdy \\
&=&\left( \int_{0}^{\infty }y^{2H-2}e^{-\rho y}dy\right) ^{2}-\left(
\int_{0}^{\infty }y^{2H-2}e^{-\theta y}dy\right) \left( \int_{0}^{\infty
}w^{2H-2}e^{-\rho w}dw\right) \\
&&+\int_{0}^{\infty }\int_{w}^{\infty }(e^{-\rho (y-w)}-e^{-\theta
(y-w)})y^{2H-2}w^{2H-2}dydw \\
&=&\Gamma (2H-1)^{2}\rho ^{1-2H}\left[ \rho ^{1-2H}-\theta ^{1-2H}\right] \\
&&+\int_{0}^{\infty }\int_{0}^{\infty }(e^{-\rho x}-e^{-\theta
x})(x+w)^{2H-2}w^{2H-2}dxdw.
\end{eqnarray*}%
Using $(x+w)^{2H-2}=\frac{1}{\Gamma (2-2H)}\int_{0}^{\infty }\xi
^{1-2H}e^{-\xi (w+x)}d\xi $, the term $A^{1}$ becomes
\begin{eqnarray*}
A^{1}(\rho ,\theta ) &=&\Gamma (2H-1)^{2}\rho ^{1-2H}\left[ \rho
^{1-2H}-\theta ^{1-2H}\right] \\
&&+\frac{\Gamma (2H-1)}{\Gamma (2-2H)}\int_{0}^{\infty }\int_{0}^{\infty
}\xi ^{2-4H}e^{-\xi x}(e^{-\rho x}-e^{-\theta x})(d\xi dx \\
&=&\Gamma (2H-1)^{2}\rho ^{1-2H}\left[ \rho ^{1-2H}-\theta ^{1-2H}\right] \\
&&+\frac{\Gamma (2H-1)\Gamma (3-4H)\Gamma (4H-2)}{\Gamma (2-2H)}\left[ \rho
^{2-4H}-\theta ^{2-4H}\right] .
\end{eqnarray*}%
Similarly, we obtain
\begin{eqnarray*}
A^{2}(\rho ,\theta ) &=&\int_{0}^{\infty }\int_{0}^{\infty }e^{-\rho
(y+w)}y^{2H-2}w^{2H-2}dwdy \\
&&+\int_{0}^{\infty }\int_{0}^{y}e^{-\rho (y-w)}y^{2H-2}w^{2H-2}dwdy \\
&&+\int_{0}^{\infty }\int_{y}^{\infty }e^{\theta (y-w)}y^{2H-2}w^{2H-2}dwdy
\\
&=&\rho ^{2-4H}\Gamma (2H-1)^{2} \\
&&+\frac{\Gamma (2H-1)\Gamma (3-4H)\Gamma (4H-2)}{\Gamma (2-2H)}\left[ \rho
^{2-4H}+\theta ^{2-4H}\right] .
\end{eqnarray*}%
Thus, $A_{\infty }(\theta ,\rho )$ is also obtained. \newline
Consequently
\begin{eqnarray*}
lim_{T\rightarrow \infty }\alpha _{H}^{2}\frac{I_{T}^{3}}{T} &=&\alpha
_{H}^{2}(A_{\infty }(\rho ,\theta )+A_{\infty }(\theta ,\rho ))  \notag \\
&=&2\eta _{H}l_{3}.  \label{I^3}
\end{eqnarray*}%
Finally, combining the above convegences we deduce that as $T\rightarrow
\infty $
\begin{equation*}
EG_{T}^{2}\longrightarrow (a,b)\ \Gamma \ ^{t}(a,b).
\end{equation*}%
On the other hand
\begin{eqnarray*}
D_{s}G_{T} &=&\frac{1}{\sqrt{T}}\left( \int_{0}^{s}\left( ae^{-\theta
(s-t)}+be^{-\rho (s-t)}\right) \delta B_{t}+\int_{s}^{T}\left( ae^{-\theta
(t-s)}+be^{-\rho (t-s)}\right) \delta B_{t}\right) \\
:= &&\frac{1}{\sqrt{T}}\left( X_{s}^{a,b}+Y_{s,T}^{a,b}\right) .
\end{eqnarray*}%
Hence
\begin{eqnarray*}
\Vert DG_{T}\Vert _{\mathcal{H}}^{2} &=&\frac{\alpha _{H}}{{T}}%
\int_{0}^{T}\int_{0}^{T}\left( X_{s}^{a,b}+Y_{s,T}^{a,b}\right) \left(
X_{r}^{a,b}+Y_{r,T}^{a,b}\right) |r-s|^{2H-2}dsdr \\
&=&\frac{\alpha _{H}}{{T}}\int_{0}^{T}\int_{0}^{T}\left(
X_{s}^{a,b}X_{r}^{a,b}+2X_{r}^{a,b}Y_{s,T}^{a,b}+Y_{s,T}^{a,b}Y_{r,T}^{a,b}%
\right) |r-s|^{2H-2}dsdr \\
:= &&\frac{\alpha _{H}}{{T}}\left(
A_{T}^{a,b}+B_{T}^{a,b}+C_{T}^{a,b}\right) .
\end{eqnarray*}%
Since $X_{s}^{a,b}$ belongs to the first Wiener chaos of $B^{H}$,
\begin{equation*}
E\left( \left\vert A_{T}^{a,b}-EA_{T}^{a,b}\right\vert ^{2}\right)
=2%
\int_{[0,T]^{4}}E(X_{s}^{a,b}X_{r}^{a,b})E(X_{u}^{a,b}X_{v}^{a,b})|u-r|^{2H-2}|v-s|^{2H-2}dsdrdudv.
\end{equation*}%
Using similar arguments as in \cite[Lemma 5.4 of web-only Appendix]{HN},
\begin{eqnarray*}
E\left( \left\vert A_{T}^{a,b}-EA_{T}^{a,b}\right\vert ^{2}\right)
&\leqslant
&c(H,\theta,\rho)%
\int_{[0,T]^{4}}|s-r|^{2H-2}|v-u|^{2H-2}|u-r|^{2H-2}|v-s|^{2H-2}dsdrdudv \\
&\leqslant &\frac{c(H,\theta,\rho)}{T^{4-8H}}%
\int_{[0,1]^{4}}|s-r|^{2H-2}|v-u|^{2H-2}|u-r|^{2H-2}|v-s|^{2H-2}dsdrdudv.
\end{eqnarray*}%
Using the same argument for $B_{T}^{a,b}$ and $C_{T}^{a,b}$ we conclude that
\begin{eqnarray*}
&&E\left( \left\vert \Vert DG_{T}\Vert _{\mathcal{H}}^{2}-E\Vert DG_{T}\Vert
_{\mathcal{H}}^{2}\right\vert ^{2}\right) \\
&\leqslant &\frac{c(H,\theta,\rho)}{T^{6-8H}}%
\int_{[0,1]^{4}}|s-r|^{2H-2}|v-u|^{2H-2}|u-r|^{2H-2}|v-s|^{2H-2}dsdrdudv \\
&\longrightarrow &0
\end{eqnarray*}%
as $T\longrightarrow \infty $, because $H<\frac{3}{4}$. This completes the
proof of Theorem \ref{key asymptotic normality}.
\end{proof}

\end{document}